\documentclass{article}

\usepackage[utf8]{inputenc}
\usepackage{srcltx}
\usepackage{amssymb,amsmath,amsfonts,amsthm}
\usepackage{cite}
\usepackage{verbatim}
\usepackage{graphics}

\usepackage[english]{babel}
\usepackage{amsmath}
\usepackage{longtable}
\usepackage{tikz}
\usepackage{amssymb}
\usepackage{amsfonts}
\usepackage{textcomp}
\usepackage{hyperref}
\usepackage{enumerate}
\usepackage{float}

\newtheorem{theorem}{Theorem}[section]
\newtheorem{remark}[theorem]{Remark}
\newtheorem{lemma}[theorem]{Lemma}
\newtheorem{proposition}[theorem]{Proposition}

\newcommand{\F}{\mathbb{F}}

\theoremstyle{definition}

\numberwithin{equation}{section}

\DeclareMathOperator{\SL}{SL}
\DeclareMathOperator{\Oo}{O}

\DeclareMathOperator{\sgn}{sgn}

\newcommand*{\No}{\textnumero}

\title{On splitting of the normalizers of maximal tori in $E_7(q)$ and $E_8(q)$}

\author{Alexey Galt and Alexey Staroletov
\thanks{Corresponding author, Sobolev Institute of Mathematics, 4 Acad. Koptyug avenue, 630090 Novosibirsk Russia, staroletov@math.nsc.ru} }

\date{\vspace{-5ex}}
\begin{document}
\newcommand{\Addresses}{{
  \bigskip
  \footnotesize

  A.~Galt, \textsc{Sobolev Institute of Mathematics, Novosibirsk, Russia;}\par\nopagebreak
  \textsc{Novosibirsk State University, Novosibirsk, Russia;}\par\nopagebreak
  \textit{E-mail address: } \texttt{galt84@gmail.com}

  \medskip

  A.~Staroletov, \textsc{Sobolev Institute of Mathematics, Novosibirsk, Russia;}\par\nopagebreak
  \textsc{Novosibirsk State University, Novosibirsk, Russia;}\par\nopagebreak
  \textit{E-mail address: } \texttt{staroletov@math.nsc.ru}
}}


\maketitle

\begin{abstract}
Let $G$ be a finite group of Lie type $E_7$ or $E_8$ over $\F_q$  and $W$ be the Weyl group of $G$. We describe all maximal tori $T$ of $G$ such that $T$ has a complement in its algebraic normalizer $N(G,T)$. Let $T$ correspond to an element $w$ of $W$. When $T$ does not have a complement, we show that $w$ has a lift to $N(G,T)$ of order $|w|$ in all considered groups, except the simply-connected group $E_7(q)$. In the latter case we describe the elements $w$ that have a lift to $N(G,T)$ of order $|w|$. 
\end{abstract}

{\bf keywords:} finite group of Lie type, maximal torus, algebraic normalizer, Weyl group

\section{Introduction}
Let $\overline{G}$ be a simple connected linear algebraic group over an algebraic closure $\overline{\F}_p$ of a finite field of positive characteristic $p$. Consider a Steinberg endomorphism $\sigma$ and a maximal $\sigma$-stable torus $\overline{T}$  of $\overline{G}$. It is well known that all maximal tori are conjugate in $\overline{G}$ and the quotient $N_{\overline{G}}(\overline{T})/\overline{T}$ is isomorphic to the Weyl group $W$ of $\overline{G}$. 

The natural question is to describe groups $\overline{G}$ such that $N_{\overline{G}}(\overline{T})$ splits over $\overline{T}$.
A similar question can be formulated for finite groups of Lie type. More precisely,
let $G$ be a finite group of Lie type, that is $O^{p'}(\overline{G}_{\sigma})\leqslant G\leqslant\overline{G}_{\sigma}$. Let $T=\overline{T}\cap G$ be a maximal torus in $G$ and $N(G,T)=N_{\overline{G}}(\overline{T})\cap G$ be the algebraic normalizer of $T$. Then the question is to describe groups $G$ and their maximal tori $T$ such that $N(G,T)$ splits over $T$.

These questions were stated by J.\,Tits in~\cite{Tits}. In the case of algebraic groups it was solved independently in~\cite{AdamsHe} and in~\cite{Galt1,Galt2,Galt3,Galt4}. In the case of finite groups the problem was
studied  for the groups of Lie types $A_n, B_n, C_n, D_n$, and $E_6$ in~\cite{Galt2,Galt3,Galt4,GS}.

J.\,Adams and X.\,He~\cite{AdamsHe} considered a related problem. Namely, it is natural to ask about the orders of lifts of $w\in W$ to $N_{\overline{G}}(\overline{T})$. They noticed that if $d$ is the order of $w$ then the minimal order of a lift of $w$ is either $d$ or $2d$, but it can be a subtle question which holds. Clearly, if $N_{\overline{G}}(\overline{T})$ splits over $\overline{T}$ then the minimal order is equal to $d$. 
In the cases of Lie types $E_6$, $E_7$, and $E_8$ they proved that the normalizer does not split and the minimal order of lifts of $w$ is $d$ if $w$ belongs to so-called regular or elliptic conjugacy classes.
G.\,Lusztig \cite{Lusztig} showed that every involution of the Weyl group of a split reductive group over ${\F}_q$ has a lift $n$  such that the image of $n$ under the Frobenius
map is equal to $n^{-1}$.

In this paper we consider finite groups $G$ of Lie types $E_7$ and $E_8$ over a finite field $\F_q$ of characteristic~$p$. Let $W$ be the Weyl group of type $E_l$, where $6\leqslant l\leqslant8$, and $\Delta=\{r_1,\ldots,r_l\}$ be a fundamental system of a root system $E_l$.
We enumerate conjugacy classes and roots of $W$ as in~\cite{DerF}. Denote by $w_i$ the element of $W$ corresponding to the reflection in the hyperplane orthogonal to the $i$-th positive root $r_i$. 

 For completeness, we first formulate the statement for groups of Lie type $E_6$.
Following~\cite{DerF}, we suppose that $r_{14}=r_2+r_4+r_5$ and $r_{36}=r_1+2r_2+2r_3+3r_4+2r_5+r_6$. Denote by $E^-_6(q)$ the finite group $^2E_6(q)$ and by $E^+_6(q)$ the finite group $E_6(q)$ (in both cases groups can be adjoint or simply-connected).  
As a consequence of ~\cite[Theorem~1.1]{GS} we obtain the following.
\begin{theorem}\label{th:E6}
Let $G=E_6^\varepsilon(q)$ (adjoint or simply-connected), where $\varepsilon\in\{+,-\}$,
and $W$ be the Weyl group of $G$. Consider a maximal torus $T$ of $G$ corresponding to an element $w$ of $W$. Then the following statements hold. 
\begin{itemize}
   \item[{\em (i)}] $T$ does not have a complement in $N(G,T)$ if and only if either
  $q\equiv-\varepsilon1\pmod{4}$ and $w$ is conjugate to $w_3w_2w_4w_{14}$
  or $q$ is odd and $w$ is conjugate to one of the following elements:
$1$, $w_1$, $w_1w_2$, $w_2w_3w_5$, $w_1w_3w_4$, $w_1w_4w_6w_{36}$, $w_1w_4w_6w_3$, $w_1w_4w_6w_3w_{36}$;
    \item[{\em (ii)}] there exists a lift of $w$ to $N(G,T)$ of order $|w|$.
  \end{itemize}
\end{theorem}

According to~\cite{DerF}, if $l=7$ then $r_{16}=r_2+r_4+r_5$ and $r_{53}=r_1+2r_2+2r_3+3r_4+2r_5+r_6$.
Denote by $U$ the subgroup of $W$ generated by $w_1,w_2,w_3,w_4,w_5,w_6$.
Note that $U$ is isomorphic to the Weyl group of type $E_6$. Given a root system $\Phi$, we denote by $\Phi^{sc}(q)$ and $\Phi^{ad}(q)$ a simply-connected and adjoint group, respectively.

\begin{theorem}\label{th:E7}
Let $G=E_7(q)$ (adjoint or simply-connected) with the Weyl group $W$ and $w_0$ be the central involution of $W$. Consider a maximal torus $T$ of $G$ corresponding to an element $w$ of $W$. Then the following statements hold.
\begin{itemize}
  \item[{\em (i)}] If $G=E_7^{ad}(q)$ then $w$ has a lift of order $|w|$ to $N(G,T)$.
  Moreover, $T$ does not have a complement in~$N(G,T)$ if and only if $q$ is odd and $w$ or $ww_0$ is conjugate to one of the following:
  1, $w_1$, $w_1w_2$, $w_2w_3w_5$, $w_1w_3w_4$, $w_1w_4w_6w_{53}$, $w_1w_4w_6w_3$, 
$w_3w_2w_4w_{16}$, $w_1w_4w_6w_3w_{53}$,  $w_3w_2w_4w_{16}w_{7}$. 
  \item[{\em (ii)}] If $G=E_7^{sc}(q)$ then $T$ does not have a complement in~$N(G,T)$ if and only if $q$ is odd. Moreover, $w$ has a lift to $N(G,T)$ of order $|w|$ if and only if either $q$ is even or $w$ satisfies one of the following:
  \begin{itemize}
  \item[{\em (a)}] $|w|$ is divisible by $4$;
  \item[{\em (b)}] $|w|$ is odd;
  \item[{\em (c)}] $w$ is conjugate in $W$ to an element of $U$.
   \end{itemize}
  \end{itemize}
\end{theorem}
\noindent{\bf Remark.} {\it In fact, it follows from Table~\ref{t:main:E7} that
the order of an element $w\in W$ divides 4 if and only if $w$ or $ww_0$
is conjugate to one of the elements listed in the claim $(i)$.}

According to~\cite{DerF}, if $l=8$ then $r_{18}=r_2+r_4+r_5$, $r_{26}=r_2+r_4+r_5+r_6$, $r_{46}=r_1+r_2+r_3+r_4+r_5+r_6+r_7$, $r_{69}=r_1+2r_2+2r_3+3r_4+2r_5+r_6$, $r_{74}=r_2+r_3+2r_4+2r_5+2r_6+2r_7+r_8$, and $r_{120}=2r_1+3r_2+4r_3+6r_4+5r_5+4r_6+3r_7+2r_8$.

\begin{theorem}\label{th:E8}
Let $G=E_8(q)$ with the Weyl group $W$ and $w_0$ be the central involution of $W$.
Consider a maximal torus $T$ of $G$ corresponding to an element $w$ of $W$. Then the following statements hold.
\begin{itemize}
  \item[{\em (i)}] $T$ does not have a complement in~$N(G,T)$ if and only if $q$ is odd and $w$ or $ww_0$ is conjugate to one of the following: 1, $w_1$, $w_1w_2$, $w_3w_1$, $w_2w_3w_5$, $w_1w_3w_5$, $w_1w_3w_4$, $w_1w_4w_6w_{69}$, $w_1w_2w_3w_5$, $w_1w_4w_6w_3$, 
$w_3w_2w_5w_4$, $w_3w_2w_4w_{18}$, $w_1w_4w_6w_3w_{69}$, $w_2w_5w_3w_4w_6$, $w_{26}w_5w_4w_3w_2$, $w_1w_4w_6w_3w_7$, $w_3w_2w_4w_{18}w_7$,  $w_{46}w_3w_5w_1w_4w_6$, $w_2w_3w_5w_7$, $w_{74}w_3w_2w_5w_4$, $w_8w_1w_4w_6w_3$, $w_1w_2w_3w_6w_8w_7$,
$w_1w_4w_3w_7w_6w_8$, $w_4w_8w_2w_5w_7w_{120}$, $w_2w_3w_4w_8w_7w_{18}$, $w_2w_3w_4w_5w_6w_8$, $w_2w_4w_5w_6w_7w_8w_{120}$,
$w_2w_3w_4w_7w_{120}w_8w_{18}$, or $w_2w_3w_4w_7w_{120}w_{18}w_8w_{74}$;
  \item[{\em (ii)}] the element $w$ has a lift to $N(G,T)$ of order $|w|$.
  \end{itemize}
\end{theorem}

We illustrate the results of these theorems in Tables~\ref{t:main:E7}, \ref{t:main:E8}, and \ref{t:main} with some additional information about the maximal tori. 

This paper is organized as follows. In Section 2 we recall notation and basic facts about algebraic groups. In Section 3 we prove auxiliary results and explain how we use MAGMA in the proofs. Section 4--6 are devoted to the proofs of the main results. 

\section{Notation and preliminary results}

By $q$ we always denote some power of a prime $p$, $\overline{\F}_p$ is an algebraic closure of a finite field $\F_p$ of order $p$. The symmetric group on $n$ elements is denoted by $S_n$ and the cyclic group of order $n$ by $\mathbb{Z}_n$. Following~\cite{CarSG}, we write $x^y=yxy^{-1}$ and $[x,y]=y^xy^{-1}$. 

By $\overline{G}$ we denote a simple connected linear algebraic group over $\overline{\F}_p$ 
with the root system $\Phi$ of Lie type $E_l$, where $l\in\{6,7,8\}$. We assume that 
$\Delta=\{r_1,\ldots,r_l\}$ is a fundamental system of $\Phi$.

We use notation from~\cite{CarSG}, in particular, the definitions of elements $x_r(t)$,
$n_r(\lambda)$ ($r\in\Phi$, $t\in\overline{\mathbb{F}}_p$, $\lambda\in\overline{\mathbb{F}}_p^\ast$). In contrast to~\cite{CarSG}, the Magma convention is that $h_r(\lambda) = n_r(-1)n_r(\lambda)$ and we will follow this definition. According to~\cite{CarSG}, $\overline{G}$ is generated by $x_r(t)$: $\overline{G}=\langle x_r(t)~|~r\in\Phi, t\in\overline{\mathbb{F}}_p\rangle$. The group 
$\overline{T}=\langle h_r(\lambda)~|~r\in\Delta,\lambda\in\overline{\mathbb{F}}_p^\ast \rangle$ is a maximal torus of $\overline{G}$ and $\overline{N}=\langle\overline{T}, n_r~|~r\in\Delta \rangle$, where $n_r=n_r(1)$, is the normalizer of $\overline{T}$ in $\overline{G}$ \cite[\S 7.1, 7.2]{CarSG}. By $W$ we denote the Weyl group $\overline{N}/\overline{T}$ and by $\pi$ the natural homomorphism from $\overline{N}$ onto $W$. Throughout, we denote by $\sigma$
the classical Frobenius automorphism defined on the generators as follows:
$$ x_r(t)\mapsto x_r(t^q), r\in\Phi, t\in\overline{\mathbb{F}}_p^\ast. $$
In particular, we have $n_r^\sigma=n_r$ and $(h_r(\lambda))^\sigma=h_r(\lambda^\sigma)$.
Define the action of $\sigma$ on $W$ in the natural way. Elements $w_1, w_2\in W$ are called {\it $\sigma$-conjugate} if $w_1=(w^{-1})^{\sigma}w_2w$ for some $w$ in $W$.
The following statement holds for $G=\overline{G}_\sigma$.


\begin{proposition}{\em\cite[Propositions 3.3.1, 3.3.3]{Car}}\label{torus}.
A torus $\overline{T}^g$ is $\sigma$-stable if and only if $g^{\sigma}g^{-1}\in\overline{N}$. The map $\overline{T}^g\mapsto\pi(g^{\sigma}g^{-1})$ determines a bijection between the $G$-classes of $\sigma$-stable maximal tori of $\overline{G}$ and the $\sigma$-conjugacy classes of $W$.
\end{proposition}

It follows from Proposition~\ref{torus} that the cyclic structure of a torus $(\overline{T}^g)_{\sigma}$ in $G$ and of the corresponding tori
in the sections of $G$ is determined only by a $\sigma$-conjugacy class of the element $\pi(g^{\sigma}g^{-1})$.

\begin{proposition}{\em\cite[Lemma 1.2]{ButGre}}\label{prop2.5}.
Let $n=g^{\sigma}g^{-1}\in\overline{N}$. Then $(\overline{T}^g)_\sigma=(\overline{T}_{\sigma n})^g$, where $n$ acts on $\overline{T}$ by conjugation.
\end{proposition}

\begin{proposition}{\em\cite[Proposition 3.3.6]{Car}}\label{p:normalizer}.
Let $g^{\sigma}g^{-1}\in\overline{N}$ and $\pi(g^{\sigma}g^{-1})=w$. Then $$(N_{\overline{G}}({\overline{T}}^g))_{\sigma}/({\overline{T}}^g)_{\sigma}\simeq C_{W,\sigma}(w)=\{x\in W~|~(x^{-1})^{\sigma}wx=w\}.$$
\end{proposition}

By Proposition~\ref{prop2.5}, we have $({\overline{T}}^g)_{\sigma}=(\overline{T}_{\sigma n})^g$ and $(N_{\overline{G}}({\overline{T}}^g))_{\sigma}=(\overline{N}^g)_{\sigma}=(\overline{N}_{\sigma n})^g$. Hence, 

$$C_{W,\sigma}(w)\simeq (N_{\overline{G}}({\overline{T}}^g))_{\sigma}/({\overline{T}}^g)_{\sigma}=(\overline{N}_{\sigma n})^g/(\overline{T}_{\sigma n})^g\simeq\overline{N}_{\sigma n}/\overline{T}_{\sigma n}.$$

\begin{remark}\label{r:nonsplit}
Let $n$ and $w$ be as in Propositions~\ref{prop2.5} and~\ref{p:normalizer}, respectively. Suppose that $n_1=g_1^{\sigma}g_1^{-1}\in\overline{N}$ and $\pi(n_1)=w$. Since $n$ and $n_1$ act on $\overline{T}$ in the same way, we have
$({\overline{T}}^{g_1})_{\sigma}=(\overline{T}_{\sigma n_1})^{g_1}=(\overline{T}_{\sigma n})^{g_1}$ and $(N_{\overline{G}}({\overline{T}}^{g_1}))_{\sigma}=(\overline{N}^{g_1})_{\sigma}=(\overline{N}_{\sigma n_1})^{g_1}$.
Hence, 
$$C_{W,\sigma}(w)\simeq (N_{\overline{G}}({\overline{T}}^{g_1}))_{\sigma}/({\overline{T}}^{g_1})_{\sigma}=(\overline{N}_{\sigma n_1})^{g_1}/(\overline{T}_{\sigma n})^{g_1}\simeq\overline{N}_{\sigma n_1}/\overline{T}_{\sigma n}.$$
Thus, $(\overline{T}^g)_\sigma$ has a complement in its algebraic normalizer if and only if there exists a complement for $\overline{T}_{\sigma n}$ in $\overline{N}_{\sigma n_1}$ for some $n_1$ with $\pi(n_1)=w$. Similarly, if $w$ has a lift to $\overline{N}_{\sigma n_1}$ of order $|w|$
then $w$ has such a lift to $N(G,T)$ as well.
\end{remark}

\begin{lemma}\label{l:ww0}
Let $w\in W$ and $w_0\in Z(W)$. Let $T$, $T_0$ be maximal tori corresponding to $w$, $ww_0$, respectively.
Then $T$ has a complement in its algebraic normalizer if and only if $T_0$ does.
\end{lemma}

\begin{proof}
Assume that $T=\overline{T}_{\sigma n}$ has a complement $K$ in $\overline{N}_{\sigma n}$. Then $K\simeq C_{W,\sigma}(w)$ and $w_0\in C_{W,\sigma}(w)$. Let $n,n_0$ be preimages of $w, w_0$ in $K$, respectively. For any $n_1\in K\leqslant\overline{N}_{\sigma n}$ we have $[n_0,n_1]=1$. It follows that $(n_1)^{\sigma nn_0}=n_1^{n_0}=n_1$ and $n_1\in\overline{N}_{\sigma nn_0}$. Hence, $K\leqslant\overline{N}_{\sigma nn_0}$. 
Since $\pi(n)=w$, we have $\pi(K)=\pi(\overline{T}_{\sigma n}\cdot K)= \pi(\overline{N}_{\sigma n})= C_{W,\sigma}(w)$. Therefore, $\pi|_K:K\rightarrow W$ is an embedding and $K\cap\overline{T}=1$. In particular, $K\cap\overline{T}_{\sigma nn_0}=1$.
Since $\overline{N}_{\sigma nn_0}/\overline{T}_{\sigma nn_0}\simeq C_{W,\sigma}(ww_0)=C_{W,\sigma}(w)$, we have $K$ is a complement for $\overline{T}_{\sigma nn_0}$ in $\overline{N}_{\sigma nn_0}$.
\end{proof}

For simplicity of notation, we write $h_r$ for $h_r(-1)$. If $r=r_i$ then $h_i$ stands for $h_{r_i}$ and $n_i$ stands for $n_{r_i}$. Every element $H$ of $\overline{T}$ can be uniquely written in the form $\prod\limits_{i=1}^lh_{r_i}(\lambda_i)$ and we write briefly $H= (\lambda_1,\lambda_2,\ldots,\lambda_l)$.

The group $\mathcal{T}=\langle n_r~|~ r\in\Delta\rangle$ is called {\it the Tits group} \cite{AdamsHe}. Denote $\mathcal{H}=\overline{T}\cap\mathcal{T}$. It is known
that $\mathcal{H}=\langle h_r~|~r\in\Delta\rangle$ and so $\mathcal{H}$ is an elementary abelian group such that $\mathcal{T}/\mathcal{H}\simeq W$. 

\begin{remark}\label{r:p=2}
Observe that if $p=2$ then $h_r=1$ for every $r\in\Delta$, in particular, $\mathcal{H}=1$ and $\mathcal{T}\simeq W$. Moreover, the restriction of the homomorphism $\pi$ to $\mathcal{T}$ is an isomorphism between $\mathcal{T}$ and $W$. Let $T$ be a maximal torus corresponding to $w\in W$. Then $n=\pi^{-1}(w)$ is a lift of $w$ to $\overline{N}_{\sigma n}$ of the same order and $\pi^{-1}(C_W(w))$ is a complement for $\overline{T}_{\sigma n}$ in $\overline{N}_{\sigma n}$. Therefore, Theorems~\ref{th:E6}-\ref{th:E8} hold in this case by Remark~\ref{r:nonsplit}.
\end{remark}

Similarly to~\cite[Theorem 7.2.2]{CarSG}, we have:
\begin{center}
$n_s n_r n_s^{-1}=n_{w_s(r)}(\eta_{s,r}),\quad \eta_{s,r}=\pm1,$
\end{center}
\begin{center}
$n_s h_r(\lambda)n_s^{-1}=h_{w_s(r)}(\lambda).$
\end{center}

We choose values of $\eta_{r,s}$ as follows. 
Let $r\in\Phi$ and $r=\sum\limits_{i=1}^l\alpha_i r_i$. The sum of the coefficients $\sum\limits_{i=1}^l\alpha_i$ is called the {\it height} of $r$. According to \cite{Vavilov}, we fix the following total ordering of positive roots: we write $r\prec s$ if either $h(r)<h(s)$ or 
$h(r)=h(s)$ and the first nonzero coordinate of $s-r$ is positive.
The table of positive roots with respect to this ordering can be found in \cite{Vavilov}.

Recall that a pair of positive roots $(r,s)$ is called {\it special} if $r+s\in\Phi$ and $r\prec s$. 
A pair $(r,s)$ is called {\it extraspecial} if it is special and for any special pair $(r_1, s_1)$ such that $r+s=r_1+s_1$ one has $r\preccurlyeq r_1$. Let $N_{r,s}$ be the structure constants of the corresponding simple Lie algebra~\cite[Section 4.1]{CarSG}. Then the signs of $N_{r,s}$ may be taken arbitrarily at the extraspecial pairs and then all other structure constants are uniquely determined \cite[Proposition~4.2.2]{CarSG}. In our case we choose $\sgn(N_{r,s})=+$ for all extraspecial pairs $(r,s)$.
Values of the structure constants for all pairs can be found in \cite{Vavilov}.
The numbers $\eta_{r,s}$ are uniquely determined by the structure constants \cite[Proposition~6.4.3]{CarSG}.

\section{Preliminaries: calculations}

We use MAGMA~\cite{MAGMA} to calculate products of elements in $\overline{N}$. 
All calculations can be performed using online Magma Calculator~\cite{MC} as well.
At the moment it uses Magma V2.25-5. We use the following preparatory commands:\newline
$L:=LieAlgebra("E7", Rationals());$\newline
$R:=RootDatum(L);$\newline
$B:=ChevalleyBasis(L);$

By default, $L$ has adjoint type. To switch data to simply-connected groups
one can replace the first line on the following:
$$L:=LieAlgebra("E7", Rationals() : Isogeny := "SC").$$

The following command produces the list of extraspecial pairs and signs of the corresponding structure constants:\newline
$x, y, h:= ChevalleyBasis(L); IsChevalleyBasis(L, R, x, y, h);$\newline
[$\langle1, 3, 1\rangle$, $\langle1, 10, 1\rangle$, $\langle1, 15, 1\rangle$, $\langle1, 17, 1\rangle$, $\langle1, 22, 1\rangle$, $\langle1, 24, 1\rangle$,
$\langle1, 28, 1\rangle$, $\langle1, 29, 1\rangle$, $\langle1, 31, 1\rangle$, $\langle1, 35, 1\rangle$, $\langle1, 36, 1\rangle$, $\langle1, 40, 1\rangle$, 
$\langle1, 41, 1\rangle$, $\langle1, 45, 1\rangle$, $\langle1, 49, 1\rangle$, $\langle1, 62, 1\rangle$, $\langle2, 4, 1\rangle$, $\langle2, 10, 1\rangle$, $\langle2, 11, 1\rangle$, 
$\langle2, 17, 1\rangle$, $\langle2, 18, 1\rangle$, $\langle2, 24, 1\rangle$, $\langle2, 25, 1\rangle$, $\langle2, 31, 1\rangle$, $\langle2, 50, 1\rangle$, $\langle2, 54, 1\rangle$,
$\langle2, 57, 1\rangle$, $\langle2, 59, 1\rangle$, $\langle3, 4, 1\rangle$, $\langle3, 11, 1\rangle$, $\langle3, 18, 1\rangle$, $\langle3, 25, 1\rangle$, 
$\langle3, 32, 1\rangle$, $\langle3, 38, 1\rangle$, $\langle3, 43, 1\rangle$, $\langle3, 44, 1\rangle$, $\langle3, 48, 1\rangle$, $\langle3, 52, 1\rangle$, $\langle3, 61, 1\rangle$, 
$\langle4, 5, 1\rangle$, $\langle4, 12, 1\rangle$, $\langle4, 19, 1\rangle$, $\langle4, 22, 1\rangle$, $\langle4, 29, 1\rangle$, $\langle4, 36, 1\rangle$, $\langle4, 46, 1\rangle$,
$\langle4, 51, 1\rangle$, $\langle4, 55, 1\rangle$, $\langle4, 60, 1\rangle$, $\langle5, 6, 1\rangle$, $\langle5, 13, 1\rangle$, $\langle5, 35, 1\rangle$, 
$\langle5, 41, 1\rangle$, $\langle5, 57, 1\rangle$, $\langle6, 7, 1\rangle$, $\langle6, 45, 1\rangle]$.

Here, for example, the triple $\langle2,4,1\rangle$ means that the pair $(r_2,r_4)$ is extraspecial and $N_{r_2,r_4}=1$.
It is straightforward to verify the defined above ordering gives the same set of extraspecial pairs.
Thus, calculations in MAGMA for $\overline{N}$ correspond to the ordering and structure constants defined in the previous section. The following commands construct elements $n_i$ and $h_i$. \newline
$G:=GroupOfLieType(L);$ \newline
$n:=[ elt\langle G~|~i\rangle : i\ in\ [1..NumberOfPositiveRoots(R)]];$ \newline
$h:=[TorusTerm(G, i,-1) : i\ in\ [1..NumberOfPositiveRoots(R)]];$

To obtain the list of matrices of reflections one can use the following command:\newline
$w:=[Transpose(i)\ :\ i\ in\ ReflectionMatrices(R)];$

\begin{lemma}{\em \cite[Lemma 1]{GS}}\label{normalizer}
Let $g\in\overline{G}$ and $n=g^\sigma g^{-1}\in\overline{N}$. Suppose that $H\in \overline{T}$ and $u\in\mathcal{T}$. Then

(i) $Hu\in\overline{N}_{\sigma n}$ if and only if   $H=H^{\sigma n}[n,u];$

(ii) If $H\in \mathcal{H}$ then $Hu\in\overline{N}_{\sigma n}$ if and only if   $[n,Hu]=1$.
\end{lemma}

The following result is our main tool for calculations of powers and commutators of elements in $\overline{N}$.

\begin{lemma}\label{conjugation} Let $\Phi$ be a simply laced root system, that is all roots in $\Phi$ have the same length. 
Let $l$ be the number of fundamental roots in $\Phi$ and $k=|\Phi|$. Assume that $n\in\mathcal{T}$ and $\pi(n)=w$. Consider the matrix $A=(a_{ij})_{l\times l}$ of $w$ in the basis $r_1,r_2,\ldots,r_l$ and an element $H=(\lambda_1,\lambda_2,\ldots,\lambda_l)$ of $\overline{T}$. Then the following statements hold.
 
(i) $H^n=(\lambda_1',\lambda_2',\ldots,\lambda_l')$,
where $\lambda_i'=\lambda_1^{a_{i1}}\cdot\lambda_2^{a_{i2}}\ldots\cdot\lambda_l^{a_{il}}$ with $1\leqslant{i}\leqslant l$;

(ii) $(Hn)^m=(\lambda_1',\lambda_2',\ldots,\lambda_l')n^m$,
where $m$ is a positive integer,  $\lambda_i'=\lambda_1^{b_{i1}}\cdot\lambda_2^{b_{i2}}\ldots\cdot\lambda_l^{b_{il}}$ with $1\leqslant{i}\leqslant l$, and $b_{ij}$ are elements of the matrix $\sum\limits_{t=0}^{m-1}A^t$.
\end{lemma}
\begin{proof} The proof is similar to that of~\cite[Lemma 2]{GS}, where this assertion was proved for $\Phi=E_6$. \end{proof}

Since we often use Lemma~\ref{conjugation}, we illustrate its applying with the following example.

\noindent{\bf Example.} Let $\Phi=E_7$, $w=w_1w_2w_3$, and $n=n_1n_2n_3$. Then it is easy to see that
$w_1(r_1)=-r_1$, $w_1(r_3)=r_1+r_2$, $w_2(r_2)=-r_2$, $w_2(r_4)=r_2+r_4$,
$w_3(r_3)=-r_3$, $w_3(r_1)=r_1+r_3$, and $w_3(r_4)=r_3+r_4$.

Therefore, in this case the matrix $A$ for $w$ is the following.
$$A=\left(\begin{array}{ccccccc}
0 &  0 & -1 & 1 & 0 & 0 & 0 \\
0 & -1 & 0  & 1 & 0 & 0 & 0 \\
1 & 0 & -1 & 1 & 0 & 0 & 0 \\
0 & 0 & 0 & 1 & 0 & 0 & 0 \\
0 & 0 & 0 & 0 & 1 & 0 & 0 \\
0 & 0 & 0 & 0 & 0 & 1 & 0 \\
0 & 0 & 0 & 0 & 0 & 0 & 1
\end{array} \right) .$$
Let $H=(\lambda_1, \lambda_2, \lambda_3, \lambda_4,\lambda_5,\lambda_6, \lambda_7)\in\overline{T}$. Then by Lemma~\ref{conjugation}, we can use the rows of $A$ to compute $H^n$, namely $H^{n}=(\lambda_3^{-1}\lambda_4,\lambda_2^{-1}\lambda_4,\lambda_1\lambda_3^{-1}\lambda_4,\lambda_4,\lambda_5,\lambda_6,\lambda_7)$.
Now let $B=A^0+A+A^2+A^3+A^4+A^5$. Then
$$B=\left(\begin{array}{ccccccc} 
0 & 0 & 0 & 2 & 0 & 0 & 0 \\
0 & 0 & 0 & 3 & 0 & 0 & 0 \\
0 & 0 & 0 & 4 & 0 & 0 & 0 \\
0 & 0 & 0 & 6 & 0 & 0 & 0 \\
0 & 0 & 0 & 0 & 6 & 0 & 0 \\
0 & 0 & 0 & 0 & 0 & 6 & 0 \\
0 & 0 & 0 & 0 & 0 & 0 & 6 
\end{array} \right) .$$
It is easy to see that $n^6=h_2$, so $$(Hn)^6=(\lambda_4^2,\lambda_4^3,\lambda_4^4,\lambda_4^6,\lambda_5^6,\lambda_6^6,\lambda_7^6)n^6=(\lambda_4^2,-\lambda_4^3,\lambda_4^4,\lambda_4^6,\lambda_5^6,\lambda_6^6,\lambda_7^6).$$

The following lemma is clear.
\begin{lemma}\label{commutator}
Suppose that $T$ and $N$ are subgroups of a group $G$ such that $T$ is abelian and 
$N\leq N_G(T)$. Let $N_1=H_1u_1, N_2=H_2u_2$, where $H_1, H_2\in T$ and $u_1,u_2\in N$. Then
$$N_1N_2=N_2N_1 \Leftrightarrow H_1^{-1}H_1^{u_2}\cdot u_2u_1u_2^{-1}u_1^{-1}=H_2^{-1}H_2^{u_1}.$$ 
In particular, if $[u_1,u_2]=1$ then
$N_1N_2=N_2N_1 \Leftrightarrow H_1^{-1}H_1^{u_2}=H_2^{-1}H_2^{u_1}.$
\end{lemma}

\section{The proof for type $E_7$}

In this section we prove Theorem~\ref{th:E7}. Since the case $p=2$ was considered in Remark~\ref{r:p=2},
we suppose that $q$ is odd. We assume that $G$ is a finite group of Lie type $E_7$.
The extended Dynkin diagram of type $E_7$ is the following.

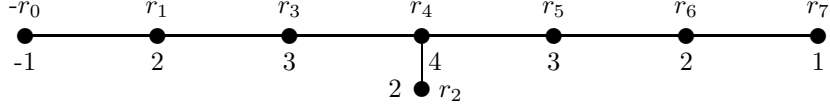
\begin{picture}(330,60)(-30,-30)
\put(50,0){\line(1,0){50}}
\put(100,0){\line(1,0){50}} \put(150,0){\line(1,0){50}}
\put(200,0){\line(1,0){50}} \put(250,0){\line(1,0){50}}
\put(50,0){\circle*{6}}
\put(100,0){\circle*{6}} \put(150,0){\circle*{6}}
\put(200,0){\circle*{6}} \put(250,0){\circle*{6}}
\put(300,0){\circle*{6}} \put(200,0){\line(0,-1){20}}
\put(200,-20){\circle*{6}} 
\put(50,10){\makebox(0,0){-$r_0$}}
\put(50,-10){\makebox(0,0){-1}}\put(100,10){\makebox(0,0){$r_1$}}
\put(100,-10){\makebox(0,0){2}}\put(150,10){\makebox(0,0){$r_3$}}
\put(150,-10){\makebox(0,0){3}}\put(200,10){\makebox(0,0){$r_4$}}
\put(205,-10){\makebox(0,0){4}}\put(250,10){\makebox(0,0){$r_5$}}
\put(250,-10){\makebox(0,0){3}}\put(300,10){\makebox(0,0){$r_6$}}
\put(300,-10){\makebox(0,0){2}}\put(211,-22){\makebox(0,0){$r_2$}}
\put(190,-20){\makebox(0,0){2}} \put(300,0){\line(1,0){50}}
\put(350,0){\circle*{6}}
\put(350,-10){\makebox(0,0){1}}\put(350,10){\makebox(0,0){$r_7$}}
\end{picture}

In this case the Weyl group $W$ is isomorphic to $2\times \Oo_7(2)$ (in the notation of \cite{Atlas}). So $W$ has the central involution $w_0$. It is easy to verify that $w_0=w_1w_2w_5w_7w_{37}w_{55}w_{61}$. There are in total 60 conjugacy classes in $W$
and all of them are divided into pairs.  Namely, for every $w\in W$ elements $w$ and $ww_0$ belong to 
different conjugacy classes. We list 30 of representatives of conjugacy classes in Table~\ref{t:main:E7} and the remaining 30 can be obtained by multiplying on $w_0$. This information can be verified with the aid of GAP \cite{GAP}.

There are exactly 60 conjugacy classes of maximal tori in $G$. The cyclic structures of maximal tori 
in groups of type $E_7$ were described in \cite{DerF}. We add this information in Table~\ref{t:main:E7}. To see the structure of a torus that corresponds to $ww_0$, where $w\in W$, one can substitute $-q$ instead of $q$ and then multiply by $-1$ in the structure of a torus that corresponds to $w$.

We divide the proof into three general cases. First, we consider the simply-connected case. Groups of adjoint type are considered in two subsections. One subsection is devoted to maximal tori that do not have complements in their algebraic normalizers and the other subsection is for the remaining maximal tori. 

\subsection{Simply-connected groups of type $E_7$}

First, we show that any maximal torus $T$ of the group $G=E_7^{sc}(q)$ does not have a complement in its algebraic normalizer $N(G,T)$. Denote $n_0=n_1n_2n_5n_7n_{37}n_{55}n_{61}$. Using MAGMA, we
verify that $n_0^2=h_2h_5h_7$, $[n_0,n_i]=1$, where $i=1,\ldots,7$, and hence $[n_0,h_i]=1$,
where $i=1,\ldots,7$. Assume that $T$ has a complement $K$ in $N(G,T)$. Since $w_0\in C_W(w)$, there exists a preimage $N_0$ of $w_0$ in $K$ such that $N_0^2=1$. Then $N_0=Hn_0$, where $H\in T$. Since $w_0$ maps $r$ to $-r$
for every $r\in\Phi$, Lemma~\ref{conjugation} implies that $H^{n_0}=H^{-1}$ for every $H\in T$. Therefore, we have $N_0^2=HH^{n_0}n_0^2=HH^{-1}h_2h_5h_7=h_2h_5h_7$, a contradiction. 

Suppose that $w$ is an element from the second column of Table~\ref{lifts_E7}. Then there exists a lift of $w$ to $N(G,T)$ of the same order. Examples of the lifts we list in the fourth column of Table~\ref{lifts_E7}.
In each case we present an element $n_w\in\mathcal{T}$ for $w\in W$ such that $\pi(n_w)=w$.
Using MAGMA, we see that $|n_w|=|w|$ and hence $n_w$ is a lift of $w$ to $\overline{N}_{\sigma{n_w}}$. 

Consider $w$ such that $|w|=4k$ for an integer $k$. 
Since the lift $N$ in Table~\ref{lifts_E7}
for $w$ lies in $\mathcal{T}$ and $n_0\in Z(\mathcal{T})$, we get $(Nn_0)^{4k}=N^{4k}n_0^{4k}=(h_2h_5h_7)^{2k}=1$. Therefore, $Nn_0$ is a lift of $ww_0$ of the same order.

Let $|ww_0|=4k+2$ for an integer $k$. Suppose that $ww_0$ has a lift $N$ to its algebraic normalizer of the same order. 
Then $N=Hnn_0=(\lambda_1,\lambda_2,\lambda_3,\lambda_4,\lambda_5,\lambda_6,\lambda_7)nn_0$ for some $\lambda_i\in\overline{\F}_p$ and $n\in\mathcal{T}$. Since $n_0$ commutes with $n$,  
we have $(nn_0)^{2i}=n^{2i}(h_2h_5h_7)^i$ and $(nn_0)^{2i+1}=n^{2i+1}n_0(h_2h_5h_7)^i$,
where $i$ is an integer. Then
\begin{multline*}
N^{4k+2}=(Hn_0n)^{4k+2}=HH^{n_0n}H^{(n_0n)^2}\ldots H^{(n_0n)^{4k+1}}(n_0n)^{4k+2}=\\
H(H^{-1})^{n}H^{n^2}(H^{-n^2})^n\ldots H^{n^{4k}}(H^{-n^{4k}})^nn^{4k+2}h_2h_5h_7.
\end{multline*}
Let $w\in\langle w_1,w_2,w_3,w_4,w_5,w_6\rangle$. This is the case for the first 25 representatives in Table~\ref{lifts_E7}. By Lemma~\ref{conjugation}, we infer that $H^{n^i}=(*,*,*,*,*,*,\lambda_7)$ for every $i>0$ and $n^{4k+2}=(*,*,*,*,*,*,1)$.
Hence, $N^{4k+2}=(*,*,*,*,*,*,-1)\neq1$; a contradiction. 

In the remaining two cases (representatives $27$ and $29$ in Table~\ref{lifts_E7}) the order of $w$ is odd, say $|w|=2k+1$. Then $(ww_0)^{2k+1}=w_0$ and $|ww_0|=4k+2$. If $ww_0$ has a preimage $N$ in $N(G,T)$ of the same order then $N^{2k+1}$ is a preimage of $w_0$ of order two, which is impossible by the argument above.

\begin{table}
\caption{Lifts of elements of $W(E_7)$ to $E_7^{sc}(q)$}\label{lifts_E7}
\centering
\begin{tabular}{|l|l|c|l|c|}
 \hline
 \No  & Representative $w$ & $|w|$ & lift of $w$  & lift of $ww_0$\\ \hline
  1  & $1$ & 1 & 1  &  -- \\ \hline
  2  & $w_1$ & 2 & $h_3n_1$  &  -- \\ \hline
  3  & $w_1w_2$ & 2 & $h_3h_4n_1n_2$   & -- \\ \hline
  4  & $w_3w_1$ & 3 & $n_3n_1$  & -- \\ \hline
  5  & $w_2w_3w_5$ & 2 & $h_4n_2n_3n_5$  & -- \\ \hline
  6  & $w_1w_3w_5$ & 6 & $h_4n_1n_3n_5$  & -- \\ \hline
  7 & $w_1w_3w_4$ & 4 & $h_2n_1n_3n_4$  & $h_2n_1n_3n_4n_0$ \\ \hline
  8 & $w_1w_4w_6w_{53}$ & 2 & $n_1n_4n_6n_{53}$  & -- \\ \hline
  9 & $w_1w_2w_3w_5$ & 6 & $h_1h_4n_1n_2n_3n_5$ & -- \\ \hline
  10 & $w_1w_5w_3w_6$ & 3 & $n_1n_5n_3n_6$ & -- \\  \hline
  11 & $w_1w_4w_6w_3$ & 4 & $h_2n_1n_4n_6n_3$  & $h_2n_1n_4n_6n_3n_0$ \\ \hline
  12 & $w_1w_4w_3w_2$ & 5 & $n_1n_4n_3n_2$  & -- \\ \hline
  13 & $w_3w_2w_5w_4$ & 6 & $n_3n_2n_5n_4$ & -- \\ \hline
  14 & $w_3w_2w_4w_{16}$ & 4 & $n_3n_2n_4n_{16}$  & $n_3n_2n_4n_{16}n_0$ \\ \hline
  15 & $w_1w_5w_3w_6w_2$ & 6 & $h_4n_1n_5n_3n_6n_2$  & -- \\ \hline
  16 & $w_1w_4w_6w_3w_{53}$ & 4 & $h_2n_1n_4n_6n_3n_{53}$  & $h_2n_1n_4n_6n_3n_{53}n_0$ \\ \hline
  17 & $w_1w_4w_5w_3w_{53}$ &  10 & $h_2n_1n_4n_5n_3n_{53}$  & -- \\ \hline
  18 & $w_1w_4w_6w_3w_5$ & 6  & $h_2n_1n_4n_6n_3n_5$  & -- \\ \hline
  19 & $w_2w_5w_3w_4w_6$ & 8  & $n_2n_5n_3n_4n_6$  & $n_2n_5n_3n_4n_6n_0$ \\ \hline
  20 & $w_{23}w_5w_4w_3w_2$ & 12 & $h_1n_{23}n_5n_4n_3n_2$  & $h_1n_{23}n_5n_4n_3n_2n_0$ \\ \hline
  21 & $w_1w_5w_2w_3w_6w_{53}$ & 3 & $n_1n_5n_2n_3n_6n_{53}$  & -- \\ \hline
  22 & $w_1w_4w_6w_3w_5w_{53}$ & 6 & $n_1n_4n_6n_3n_5n_{53}$ & -- \\ \hline
  23 & $w_1w_4w_6w_3w_2w_5$ & 12 & $n_1n_4n_6n_3n_2n_5$  & $n_1n_4n_6n_3n_2n_5n_0$ \\ \hline
  24 & $w_1w_4w_{16}w_3w_2w_6$ & 9 & $n_1n_4n_{16}n_3n_2n_6$  & -- \\ \hline
  25 & $w_1w_4w_{16}w_3w_2w_{40}$ & 6 & $n_1n_4n_{16}n_3n_2n_{40}$ & --\\ \hline
  26 & $w_1w_4w_{6}w_3w_7$ & 12 & $h_2n_1n_4n_6n_3n_7$  & $h_2n_1n_4n_6n_3n_7n_0$ \\ \hline
  27 & $w_1w_4w_{6}w_2w_3w_7$ & 15 & $n_1n_4n_6n_2n_3n_7$ & -- \\ \hline
  28 & $w_3w_2w_4w_{16}w_7$ & 4 & $n_3n_2n_4n_{16}n_7$  & $n_3n_2n_4n_{16}n_7n_0$ \\ \hline
  29 & $w_1w_4w_{6}w_3w_5w_7$ & 7 & $n_1n_4n_6n_3n_5n_7$ & -- \\ \hline
  30 & $w_{39}w_3w_5w_1w_4w_6$ & 8 & $n_{39}n_3n_5n_1n_4n_6$ & $n_{39}n_3n_5n_1n_4n_6n_0$ \\
  \hline
\end{tabular}
\end{table}

\subsection{Adjoint groups of type $E_7$: non-splitting  cases} 
Let $G=E^{ad}_7(q)$.
Throughout this subsection we suppose that $\overline{T}$ is a maximal $\sigma$-stable torus of $\overline{G}$ and $T$ is a maximal torus of $G$ corresponding to a conjugacy class of $w\in W$. 
We write $H=(\lambda_1,\lambda_2,\lambda_3,\lambda_4,\lambda_5,\lambda_6,\lambda_7)$ for an arbitrary element of $T$. This notation means that $H=\prod\limits_{i=1}^7h_{r_i}(\lambda_i)$.
Observe that $Z(E^{sc}_7(q))=\langle h_2h_5h_7\rangle$.
For simplicity, we identify $h_i$ with their images in $G$ and 
assume that $h_2h_5h_7=1$, that is $(1,-1,1,1,-1,1,-1)=(1,1,1,1,1,1,1)$. Recall that $\pi$ is the natural homomorphism from $\overline{N}$ onto $W$.
The main tool to show that $T$ does not have a complement in $N(G,T)$ is the following assertion.
\begin{lemma}\label{l:torie7}
Let $w\in W$ and $C_W(w)\geqslant\langle w_2w_5,w_{49},w_{63}\rangle \simeq\mathbb{Z}_2\times\mathbb{Z}_2\times\mathbb{Z}_2$.
Suppose that $n\in\overline{N}$ such that $\pi(n)=w$ and $n_2n_5,n_{49},n_{63}$ lie in $\overline{N}_{\sigma n}$. Then $\overline{T}_{\sigma n}$ does not have a complement in $\overline{N}_{\sigma n}$.
\end{lemma}
\begin{proof}
Assume that $\overline{T}_{\sigma n}$ has a complement $K$ in $\overline{N}_{\sigma n}$.
Let $N_1,N_2,N_3$ be preimages of $w_2w_5$, $w_{49}$, and $w_{63}$ in $K$, respectively. Then $N_1=H_1n_2n_5$, $N_2=H_2n_{49}$, $N_3=H_3n_{63}$, where
$$H_1=(\mu_1,\ldots,\mu_7),
H_2=(\alpha_1,\ldots,\alpha_7),
H_3=(\beta_1,\ldots,\beta_7)$$
are elements of $\overline{T}_{\sigma n}$.

Since $w_{63}^2=1$, it is true that $N_3^2=1$. Computations in MAGMA show
that $n_{63}^2=h_3h_5h_7$. By Lemma~\ref{conjugation}, 
$$(Hn_{63})^2=(1,\lambda_1^{-2}\lambda_2^2,-\lambda_1^{-3}\lambda_3^2,\lambda_1^{-4}\lambda_4^2,
-\lambda_1^{-3}\lambda_5^2,\lambda_1^{-2}\lambda_6^2,-\lambda_1^{-1}\lambda_7^2).$$
Therefore, $\beta_1^2=\varepsilon\beta_2^2$ and $\beta_5^2=-\varepsilon\beta_1^3$, 
where $\varepsilon\in\{1,-1\}$.

Since $[w_{63}, w_2w_5]=[w_{63}, w_{49}]=1$, we infer that $[N_3,N_1]=[N_3,N_2]=1$.
Using MAGMA, we see that $[n_{63},n_2n_5]=[n_{63},n_{49}]=1$. It follows from Lemma~\ref{commutator} that $H_3^{-1}H_3^{n_2n_5}=H_1^{-1}H_1^{n_{63}}$ and 
$H_3^{-1}H_3^{n_{49}}=H_2^{-1}H_2^{n_{63}}$.
By Lemma~\ref{conjugation},
$$H^{-1}H^{n_{63}}=(\lambda_1^{-2},\lambda_1^{-2},\lambda_1^{-3},\lambda_1^{-4}, 
\lambda_1^{-3},\lambda_1^{-2},\lambda_1^{-1}),$$
$$H^{-1}H^{n_2n_5}=(1,\lambda_2^{-2}\lambda_4,1,1,\lambda_4\lambda_5^{-2}\lambda_6,1,1),$$
$$H^{-1}H^{n_{49}}=(1,\lambda_1\lambda_6^{-1},\lambda_1\lambda_6^{-1},\lambda_1^2\lambda_6^{-2},
\lambda_1^2\lambda_6^{-2},\lambda_1^2\lambda_6^{-2},\lambda_1\lambda_6^{-1}).$$
So we have $(1,\beta_2^{-2}\beta_4,1,1,\beta_4\beta_5^{-2}\beta_6,1,1)=(\mu_1^{-2},\mu_1^{-2},\mu_1^{-3},\mu_1^{-4},\mu_1^{-3},\mu_1^{-2},\mu_1^{-1})$. 
Whence $\mu_1^{-3}=\mu_1^{-4}$ and so $\mu_1=1$. Then $\beta_2^{-2}\beta_4=\beta_4\beta_5^{-2}\beta_6=1$. Therefore, $\beta_5^2=\beta_6\beta_2^2$.
Since $\beta_1^2=\varepsilon\beta_2^2$ and $\beta_5^2=-\varepsilon\beta_1^3$, we infer that $\beta_6=-\beta_1$.
The equality $H_3^{-1}H_3^{n_{49}}=H_2^{-1}H_2^{n_{63}}$ implies that $$(1,\beta_1\beta_6^{-1},\beta_1\beta_6^{-1},\beta_1^2\beta_6^{-2},
\beta_1^2\beta_6^{-2},\beta_1^2\beta_6^{-2},\beta_1\beta_6^{-1})=(\alpha_1^{-2},\alpha_1^{-2},\alpha_1^{-3},\alpha_1^{-4},\alpha_1^{-3},\alpha_1^{-2},\alpha_1^{-1}).$$ 
Applying $\beta_1\beta_6^{-1}=-1$ to the latter equality, we find that
$$(\alpha_1^{-2},\alpha_1^{-2},\alpha_1^{-3},\alpha_1^{-4},\alpha_1^{-3},\alpha_1^{-2},\alpha_1^{-1})=(1,-1,-1,1,1,1,-1).$$
From the equalities of the first coordinates, we see that $\alpha_1^2=1$. Since on the right side the fifth and the seventh coordinates have different signs, we infer that $\alpha_1^{-3}=-\alpha_1^{-1}$; a contradiction with $\alpha_1^2=1$.
\end{proof}

The following lemma finishes the proof of the theorem for cases when $T$ does not have complements.

\begin{lemma} Let $w$ or $w_0w$ is one of the following:
1, $w_1$, $w_1w_2$, $w_2w_3w_5$, $w_1w_3w_4$, $w_1w_4w_6w_{53}$, $w_1w_4w_6w_3$, 
$w_3w_2w_4w_{16}$, $w_1w_4w_6w_3w_{53}$,  $w_3w_2w_4w_{16}w_{7}$. Suppose that $T$
is a maximal torus that corresponds to the conjugacy class of $w$ in $W$.
Then $N(G,T)$ does not split over $T$ and there exists a lift of $w$ to $N(G,T)$ of order $|w|$.
\end{lemma}
\begin{proof} By Lemma~\ref{l:ww0}, we can assume that $w$ is an element from the list in the hypothesis.
Using GAP, we find that in each case there exists $w'\in W$ which is conjugate to $w$ 
and such that $w_2w_5$, $w_{49}$, $w_{63}\in C_W(w')$. Examples of such $w'$ for every $w$
are listed in Table~\ref{t:nonsplit_E7}. The first column of this table contains the number of a torus according to Table~\ref{t:main:E7},
and the third column contains an example of $w'$ for $w$. The fourth column contains elements $n_{w'}\in\mathcal{T}$ 
such that $[n_{w'},n_2n_5]=[n_{w'},n_{49}]=[n_{w'},n_{63}]=1$ and $\pi(n_{w'})=w'$.
Then Lemma~\ref{normalizer} yields $n_2n_5$, $n_{49}$, and $n_{63}$ belong to $\overline{N}_{\sigma n_{w'}}$ in each case.
Now Lemma~\ref{l:torie7} implies that $\overline{N}_{\sigma n_{w'}}$ does not split over $\overline{T}_{\sigma n_{w'}}$. So $N(G,T)$ does not split over $T$ by Remark~\ref{r:nonsplit}.
Finally, the fifth 
column contains another preimage $n'$ of $w'$ in $\mathcal{T}$ such that $|w'|=|n'|$. The latter equality can be verified in MAGMA. Thus, $n'$ is a lift of $w'$ to $\overline{N}_{\sigma n'}$ of the same order.
Then $n'n_0$ is a lift of $w'w_0$ and the lemma is proved.
\end{proof}

\begin{longtable}{|c|l|l|l|l|}
\caption{Representatives of conjugacy classes satisfying the condition of Lemma~\ref{l:torie7}}\label{t:nonsplit_E7} \\
\hline
 \No & $w$ & $w'$ & preimage of $w'$ & lift of $w'$ \\ 
\hline\endfirsthead
\multicolumn{5}{c}{\textit{Table~\ref{t:nonsplit_E7} (continued)}} \\
\hline 
 \No & $w$ & $w'$ & preimage of $w'$ & lift of $w'$ \\ 
\hline\endhead
\hline
\endfoot
\hline\endlastfoot 
  1  & $1$ &  $1$ &  1  & 1 \\ \hline
  2  & $w_1$ &  $w_7$ & $n_7$ & $h_6n_7$ \\ \hline
  3  & $w_1w_2$ &  $w_2w_5$ & $n_2n_5$ & $h_4n_2n_5$  \\ \hline
  5  & $w_2w_3w_5$ &  $w_3w_5w_2$ & $n_3n_5n_2$  & $h_4n_3n_5n_2$ \\ \hline
  7 & $w_1w_3w_4$ &  $w_2w_4w_{16}$ & $h_4n_2n_4n_{16}$ & $h_6n_2n_4n_{16}$ \\ \hline
  8 & $w_1w_4w_6w_{53}$ &  $w_2w_5w_{49}w_{63}$ & $n_2n_5n_{49}n_{63}$ & $n_2n_5n_{49}n_{63}$ \\ \hline
  11 & $w_1w_4w_6w_3$ &  $w_2w_{63}w_4w_{16}$ & $h_4n_2n_{63}n_4n_{16}$ & $h_6n_2n_{63}n_4n_{16}$ \\ \hline
  14 & $w_3w_2w_4w_{16}$ &  $w_3w_2w_4w_{16}$ & $h_4n_3n_2n_4n_{16}$ & $n_3n_2n_4n_{16}$ \\  \hline
  16 & $w_1w_4w_6w_3w_{53}$ &  $w_2w_8w_{32}w_{57}w_{60}$ & $h_1n_2n_8n_{32}n_{57}n_{60}$ & $h_6n_2n_8n_{32}n_{57}n_{60}$ \\ \hline
  28 & $w_3w_2w_4w_{16}w_{7}$ & $w_3w_2w_4w_{16}w_7$ & $h_4n_3n_2n_4n_{16}n_7$ & $n_3n_2n_4n_{16}n_7$ \\ 
\end{longtable}

\subsection{Adjoint groups of type $E_7$: splitting cases}
In this subsection, we show that maximal tori of $G$ that are not considered in the previous subsection
have complements in their algebraic normalizers.
Recall that $w_0=w_1w_2w_5w_7w_{37}w_{55}w_{61}$ is the central involution of $W$ and $n_0=n_1n_2n_5n_7n_{37}n_{55}n_{61}$. As above, we identify $h_i$ with their images in $G$,
in particular $h_2h_5h_7=1$.

We divide the proof into two parts: first, we consider maximal tori that need a special treatment and 
then we finish with remaining maximal tori using a common approach.

\begin{lemma} Suppose that $w$ or $ww_0$ is conjugate to one
of the following: $w_3w_1$, $w_1w_3w_5$, $w_1w_2w_3w_5$, $w_2w_5w_3w_4w_6$, $w_{23}w_5w_4w_3w_2$,
$w_1w_4w_6w_3w_7$, or $w_{39}w_3w_5w_1w_4w_6$. If $T$ is a maximal torus corresponding to the conjugacy class of $w$ 
then $N(G,T)$ splits over $T$.
\end{lemma}

\begin{proof}
We consider each case separately.

\textbf{Torus 4.} In this case $w=w_3w_1$ and $C_W(w)=\langle ww_0\rangle\times\langle w_2, w_{50}, w_7, w_6, w_5 \rangle\simeq \mathbb{Z}_{6} \times S_6$.
Here we use that 
\begin{multline*}
S_6\simeq\langle a, b, c, d, e~|~a^2=b^2=c^2=d^2=e^2=(ab)^3=(bc)^3=(cd)^3=(de)^3=[a,c]=[a,d]=[a,e]=\\=[b,d]=[b,e]=[c,e]=1\rangle,
\end{multline*}
and $w_2$, $w_{50}$, $w_7$, $w_6$, and $w_5$ satisfy this set of relations.

Let $\alpha\in\overline{\F}_p$ such that $\alpha^2=-1$ and set 
$$H_2=(1,\alpha,1,1,\alpha,-1,-\alpha), H_3=(1,\alpha,1,1,-\alpha,-1,\alpha), H_4=H_5=(-1,\alpha,1,-1,\alpha,1,\alpha),$$ $$\text{and }H_6=(-1,\alpha,1,-1,\alpha,-1,\alpha).$$

Let $n=n_3n_1$, $N_1=nn_0$, $N_2=H_2n_2$, $N_3=H_3n_{50}$, $N_4=H_4n_{7}$, $N_5=H_5n_{6}$, and $N_6=H_6n_5$.
We claim that $K=\langle N_1, N_2, N_3, N_4, N_5, N_6 \rangle$ is a complement for $\overline{T}_{\sigma n}$ in $\overline{N}_{\sigma n}$.

Since $H_i^{-1}H^{n_0}=H_i^{-2}=1$ for every $i\in\{2\ldots6\}$, Lemma~\ref{commutator} implies
that $[n_0,N_2]=[n_0,N_3]=[n_0,N_4]=[n_0,N_5]=[n_0,N_6]=1$. Using MAGMA, we see that $N_1^6=1$ and $[n_3n_1,n_2]=[n_3n_1,n_{50}]=[n_3n_1,n_7]=[n_3n_1,n_6]=[n_3n_1,n_5]=1$, so $n_2$, $n_{50}$, $n_7$, $n_6$, $n_5$ belong to $\overline{N}_{\sigma n}$.

By Lemma~\ref{conjugation}, we have
$H^{n}=(\lambda_1^{-1}\lambda_3, \lambda_2, \lambda_1^{-1}\lambda_4,\lambda_4,\lambda_5,\lambda_6,\lambda_7)$.
Therefore, $H_i^{n}=H_i$ for $i\in\{2,\ldots,6\}$. This implies that $H_i^{\sigma n}=(1,(-1)^{((q-1)/2)},1,1,(-1)^{((q-1)/2)},1,(-1)^{((q-1)/2)})H_i=H_i$, where  $i\in\{2,\ldots,6\}$, and hence $H_i\in \overline{T}_{\sigma n}$.
Moreover, Lemma~\ref{commutator} yields $[n_3n_1,N_i]=1$, where $i\in\{2,\ldots,6\}$. So $K\simeq\langle N_1 \rangle\times\langle N_2,N_3,N_4,N_5,N_6\rangle $.
 
Now we prove that $N_2^2=(N_2N_3)^3=[N_2,N_4]=[N_2,N_5]=[N_2,N_6]=1$. 
Since $n_2^2=h_2$, Lemma~\ref{conjugation} implies that $(Hn_2)^2=(\lambda_1^2,-\lambda_4,\lambda_3^2,\lambda_4^2,
\lambda_5^2,\lambda_6^2,\lambda_7^2)$. Therefore, $N_2^2=h_2h_5h_7=1$.
Calculations in MAGMA show that $[n_2,n_7]=[n_2,n_6]=[n_2,n_5]=1$. By Lemma~\ref{commutator}, 
to get that $[N_2,N_4]=[N_2,N_5]=[N_2,N_6]=1$ it is sufficient to verify that $H_4^{n_2}=H_4$, $H_5^{n_2}=H_5$, $H_6^{n_2}=H_6$ and $H_2=H_2^{n_7}=H_2^{n_6}=H_2^{n_5}$.
By Lemma~\ref{conjugation}, we find that
$$H^{n_2}=(\lambda_1,\lambda_2^{-1}\lambda_4,\lambda_3,\lambda_4,\lambda_5,\lambda_6,\lambda_7),
H^{n_5}=(\lambda_1,\lambda_2,\lambda_3,\lambda_4,\lambda_4\lambda_5^{-1}\lambda_6,\lambda_6,\lambda_7),$$
$$H^{n_6}=(\lambda_1,\lambda_2,\lambda_3,\lambda_4,\lambda_5,\lambda_5\lambda_6^{-1}\lambda_7,\lambda_7),
\text{ and }H^{n_7}=(\lambda_1,\lambda_2,\lambda_3,\lambda_4,\lambda_5,\lambda_6,\lambda_6\lambda_7^{-1}).$$
Applying these equations to $H_2$, $H_4$, $H_5$, and $H_6$, we infer that $N_2$ commutes with $N_4$, $N_5$ and $N_6$.
Finally, observe that $N_2N_3=H_2H_3^{n_2}n_2n_{50}$ and $H_2H_3^{n_2}=H_2(1,-\alpha,1,1,-\alpha,-1,\alpha)=1$. So $(N_2N_3)^3=(n_2n_{50})^3=1$, as claimed.

Now we prove that $N_3^2=(N_3N_4)^3=1$ and $[N_3,N_5]=[N_3,N_6]=1$. 
Using MAGMA, we see that $n_{50}^2=h_1h_2h_4h_6$. Lemma~\ref{conjugation} implies that $$(Hn_{50})^2=(-\lambda_1^2\lambda_2\lambda_4^{-1}\lambda_7,-\lambda_2^3\lambda_4^{-1}\lambda_7,\lambda_2^2\lambda_3^2\lambda_4^{-2}\lambda_7^2,-\lambda_2^3\lambda_4^{-1}\lambda_7^3,\lambda_2^2\lambda_4^{-2}\lambda_5^2\lambda_7^2,-\lambda_2\lambda_4^{-1}\lambda_6^2\lambda_7,\lambda_7^2).$$ 
Therefore, $N_3^2=h_2h_5h_7=1$. Calculations in MAGMA show that $[n_{50},n_6]=[n_{50},n_5]=1$. 
From the above equations, we see that $H_3^{n_6}=H_3$ and $H_3^{n_5}=H_3$.
Lemma~\ref{conjugation} implies that $H^{-1}H^{n_{50}}=(t,t,t^2,t^3,t^2,t,1)$, where $t=\lambda_2\lambda_4^{-1}\lambda_7$.
So $H_5^{-1}H_5^{n_{50}}=H_6^{-1}H_6^{n_{50}}=1$. By Lemma~\ref{commutator}, we infer that $N_3$ commutes with $N_5$ and $N_6$.
Observe that $N_3N_4=H_3H_4^{n_{50}}n_{50}n_7$. We know that $H_4=H_5$ and $H_5^{n_{50}}=H_5$, so $H_3H_4^{n_{50}}=H_3H_4=h_1h_2h_4h_6h_7$. Using MAGMA, we see that $(h_1h_2h_4h_6h_7n_{50}n_7)^3=1$ and hence $(N_3N_4)^3=1$.

Now we prove that $N_4^2=(N_4N_5)^3=1$ and $[N_4,N_6]=1$. 
First, observe that $N_4^2=(H_4n_7)^2=H_4H_4^{n_7}h_7$. By the above equation for $H^{n_7}$, we infer that
$H_4^{n_7}=H_4h_7$ and so $N_4^2=1$. Since $N_4N_5=H_4H_5^{n_7}n_7n_6$ and $H_4H_5^{n_7}=H_4H_4^{n_7}=h_7$, we have $(N_4N_5)^3=(h_7n_7n_6)^3=1$. Calculations in MAGMA show that $[n_7,n_5]=1$, so $[N_4,N_6]=1$ is equivalent to
$H_4^{-1}H_4^{n_5}=H_6^{-1}H_6^{n_7}$. By the above equations for $H^{n_7}$ and $H^{n_6}$, we have 
$H_4=H_4^{n_5}$ and $H_6=H_6^{n_7}$. Therefore,  $N_4$ and $N_6$ commute.

Finally, we verify that $N_5^2=N_6^2=(N_5N_6)^3=1$. Since $N_5^2=H_5H_5^{n_6}h_6$ and 
$N_6^2=H_6H_6^{n_5}h_5$, we infer that $N_5^2=N_6^2=1$. Observe that $N_5N_6=H_5H_6^{n_6}n_6n_5=H_5(-1,\alpha,1,-1,\alpha,1,\alpha)n_6n_5=h_2h_5h_7n_6n_5=n_6n_5$, so
$(N_5N_6)^3=(n_6n_5)^3=1$. Thus, $\langle N_2, N_3, N_4, N_5, N_6\rangle\simeq S_6$ and hence $K\simeq C_W(w)$.

\textbf{Torus 6.} In this case $w=w_1w_3w_5$ is conjugate to $w'=w_1w_2w_3$ and $$C_W(w')=\langle w'\rangle\times\langle w_0\rangle\times\langle w_5, w_6,w_7\rangle\simeq \mathbb{Z}_{6}\times \mathbb{Z}_2\times S_4.$$

Let $n=n_1n_2n_3$ and $\alpha\in\overline{\F}_p$ such that $\alpha^2=-1$. Using MAGMA, we see that $[n, n_5]=[n, n_6]=[n, n_7]=1$ and $n_5,n_6,n_7\in\overline{N}_{\sigma{n}}$.
Put $$H_2=(-1,\alpha,1,-1,\alpha,-1,\alpha), H_3=H_4=(-1,\alpha,1,-1,\alpha,1,\alpha).$$
By Lemma~\ref{conjugation},
$$H^{n}=(\lambda_3^{-1}\lambda_4, \lambda_2^{-1}\lambda_4, \lambda_1\lambda_3^{-1}\lambda_4, \lambda_4,\lambda_5,\lambda_6,\lambda_7).$$
Then $H_2^{\sigma n}=(-1,\alpha,1,-1,\alpha,-1,\alpha)^{\sigma}=H_2$ and $H_2\in\overline{T}_{\sigma n}$.
It is easy to see that $h_6\in\overline{T}_{\sigma n}$ and hence $H_3,H_4\in\overline{T}_{\sigma n}$.
Put $N_2=H_2n_5, N_3=H_3n_6, N_4=H_4n_7$.
We claim that $\langle N_2,N_3, N_4\rangle\simeq S_4$.

By Lemma~\ref{conjugation}, we have
$$H^{n_5}=(\lambda_1,\lambda_2,\lambda_3,\lambda_4,\lambda_4\lambda_5^{-1}\lambda_6,\lambda_6,\lambda_7),
H^{n_6}=(\lambda_1,\lambda_2,\lambda_3,\lambda_4,\lambda_5,\lambda_5\lambda_6^{-1}\lambda_7,\lambda_7),$$
$$H^{n_7}=(\lambda_1,\lambda_2,\lambda_3,\lambda_4,\lambda_5,\lambda_6,\lambda_6\lambda_7^{-1}).$$

Then $N_2^2=H_2H_2^{n_5}n_5^2=h_2h_7h_5=1$.  Similarly, we see that $N_3^2=N_4^2=h_2h_5h_7=1$.
Moreover, we have $$N_2N_3=H_2H_3^{n_5}n_5n_6=(-1,\alpha,1,-1,\alpha,-1,\alpha)(-1,\alpha,1,-1,\alpha,1,\alpha)n_5n_6=h_6n_5n_6, \text{ and }$$
$$N_3N_4=H_3H_4^{n_6}n_6n_7=(-1,\alpha,1,-1,\alpha,1,\alpha)(-1,\alpha,1,-1,\alpha,-1,\alpha)n_6n_7=h_6n_6n_7.$$
Calculations in MAGMA show that $(h_6n_5n_6)^3=(h_6n_6n_7)^3=1$. 
Hence, $(N_2N_3)^3=(N_3N_4)^3=1$ and $\langle N_2,N_3, N_4\rangle\simeq S_4$.

Assume that $q\equiv3\pmod{4}$. Let $\beta\in\overline{\F}_p$ such that $\beta^{q+1}=-1$, 
$$H_1=(1,\beta,1,1,-\alpha,-1,\alpha), H_0=(1,\alpha\beta,1,1,-\alpha,-1,\alpha).$$
Then $\alpha^{q}=-\alpha$ and $H_1^{\sigma n}= (1,\beta^{-1},1,1,-\alpha,-1,\alpha)^{\sigma}=(1,-\beta,1,1,\alpha,-1,-\alpha)=H_1$. Therefore,  we have $H_1\in \overline{T}_{\sigma n}$.
Since $H_0=H_1(1,\alpha,1,1,1,1,1)$ and $(1,\alpha,1,1,1,1,1)^{\sigma{n}}=(1,\alpha^{-1},1,1,1,1,1,)^{\sigma}=(1,\alpha,1,1,1,1,1)$, we also get $H_0\in\overline{T}_{\sigma n}$. Put $N_1=H_1n$ and $N_0=H_0n_0$.
We claim that $K=\langle N_1,N_0,N_2,N_3,N_4\rangle$ is a complement for $\overline{T}_{\sigma n}$. 

Calculations in MAGMA show that $n^6=h_2$. By Lemma~\ref{conjugation}, we have
$$(Hn)^6=(\lambda_4^2,-\lambda_4^3,\lambda_4^4,\lambda_4^6,\lambda_5^6,\lambda_6^6,\lambda_7^6).$$
Hence, $N_1^6=h_2h_5h_7=1$. 

Using the above equations, we find that $H_1^{-1}H_1^{n_5}=H_1^{-1}H_1^{n_6}=H_1^{-1}H_1^{n_7}=1$, $H_1^{-1}H_1^{N_0}=H_1^{-2}=(1,\beta^{-2},1,1,-1,1,-1).$ 
Moreover, we see that $H_2^{-1}H_2^{n}=H_3^{-1}H_3^{n}=H_4^{-1}H_4^{n}=1$ and $H_0^{-1}H_0^{n}=(1,-\beta^{-2},1,1,1,1,1).$ 
Hence, by Lemma~\ref{commutator} we obtain $[N_1,N_2]=[N_1,N_3]=[N_1,N_4]=[N_1,N_0]=1$.

Now we see that $H_2^{-2}=H_3^{-2}=H_4^{-2}=(1,-1,1,1,-1,1,-1)=1$.
On the other hand, $H_0^{-1}H_0^{n_5}=H_0^{-1}H_0^{n_6}=H_0^{-1}H_0^{n_7}=1$ and
hence Lemma~\ref{commutator} implies that $[N_0,N_2]=[N_0,N_3]=[N_0,N_4]=1$. So
$K\simeq\langle N_1\rangle\times\langle N_0\rangle\times\langle N_2,N_3,N_4\rangle\simeq \mathbb{Z}_6\times \mathbb{Z}_2\times S_4$ is a complement for $\overline{T}_{\sigma n}$.

Assume that $q\equiv1\pmod{4}$. Let $\alpha,\delta\in\overline{\F}_p$ such that $\alpha^2=-1$ and $\delta^{q-1}=-1$.
Observe that $\alpha^q=\alpha$ and $\delta^q=-\delta$. Put 
$$H_1=(1,1,1,1,\alpha,-1,-\alpha) \text{ and } H_0=(\delta^4,\alpha\delta^6,\delta^8,\delta^{12},\delta^9,\delta^6,\delta^3).$$
Then $H_1^{\sigma n}=(1,1,1,1,\alpha,-1,-\alpha)^{\sigma}=H_1$ and so $H_1\in \overline{T}_{\sigma n}$.
Now
$$H_0^{\sigma n}=(\delta^{-8}\delta^{12}, -\alpha\delta^{6}, \delta^4\delta^{-8}\delta^{12},\delta^{12},\delta^9,\delta^6,\delta^3)^{\sigma}=
(\delta^4,-\alpha\delta^6,\delta^8,\delta^{12},-\delta^9,\delta^6,-\delta^3)=H_0$$ and hence $H_0\in\overline{T}_{\sigma n}$.

We claim that $K=\langle N_1,N_2,N_3,N_4,N_0\rangle$ is a complement for $\overline{T}_{\sigma n}$. 
As it was shown above, we have 
$(Hn)^6=(\lambda_4^2,-\lambda_4^3,\lambda_4^4,\lambda_4^6,\lambda_5^6,\lambda_6^6,\lambda_7^6).$
Therefore,  we find that $N_1^6=(1,-1,1,1,-1,1,-1)=1$. 

We see above that $H_2^{-1}H_2^{n}=H_3^{-1}H_3^{n}=H_4^{-1}H_4^{n}=1$.

Since $H^{-1}H^{n}=(\lambda_1^{-1}\lambda_3^{-1}\lambda_4, \lambda_2^{-2}\lambda_4, \lambda_1\lambda_3^{-2}\lambda_4, 1, 1, 1, 1),$ then
$$H_0^{-1}H_0^{n}=(1,-1,1,1,1,1,1)=h_2.$$ 
Recall that
$$H^{-1}H^{n_5}=(1,1,1,1,\lambda_4\lambda_5^{-2}\lambda_6,1,1), H^{-1}H^{n_6}=(1,1,1,1,1,\lambda_5\lambda_6^{-2}\lambda_7,1),
H^{-1}H^{n_7}=(1,1,1,1,1,1,\lambda_6\lambda_7^{-2}).$$
Then $H_1^{-1}H_1^{n_5}=H_1^{-1}H_1^{n_6}=H_1^{-1}H_1^{n_7}=1$ and $$H_1^{-1}H_1^{N_0}=H_1^{-2}=(1,1,1,1,-1,1,-1)=h_5h_7=h_2.$$ 

Hence, by Lemma~\ref{commutator}, we obtain $[N_1,N_2]=[N_1,N_3]=[N_1,N_4]=[N_1,N_0]=1$.

Furthermore, $H_0^{-1}H_0^{n_5}=H_0^{-1}H_0^{n_6}=H_0^{-1}H_0^{n_7}=1$ and
by Lemma~\ref{commutator} we obtain $[N_0,N_2]=[N_0,N_3]=[N_0,N_4]=1$. Hence,
$K\simeq\langle N_1\rangle\times\langle N_0\rangle\times\langle N_2,N_3,N_4\rangle\simeq \mathbb{Z}_6\times \mathbb{Z}_2\times S_4$ is a complement for $\overline{T}_{\sigma n}$.

\textbf{Torus 9.}  In this case $w=w_1w_2w_3w_5$ and $C_W(w)=\langle w_0w^2\rangle\times\langle w_7 \rangle\times\langle w_5, w_{58}w_{59} \rangle\simeq \mathbb{Z}_6\times \mathbb{Z}_2\times D_8$.

Let $n=h_4n_1n_2n_3n_5$ and $\alpha\in\overline{\F}_p$ such that $\alpha^2=-1$. Denote $N_1=n_0n^2$, $N_2=H_1n_7$,
$N_3=H_1h_6n_5$, and $N_4=H_1h_1h_4n_{58}n_{59}$, where $H_1=(-1,\alpha,1,-1,-\alpha,1,\alpha)$. We claim that
$K=\langle N_1,N_2,N_3, N_4\rangle$ is a complement for $\overline{T}_{\sigma n}$ in $\overline{N}_{\sigma n}$.

Using MAGMA, we see that $[n,n_7]=[n,h_6n_5]=[n,h_1h_4n_{58}n_{59}]=1$, so $n_7$, $h_6n_5$, and $h_1h_4n_{58}n_{59}$
belong to $\overline{N}_{\sigma n}$. Now we verify that $H_1\in \overline{T}_{\sigma n}$. By Lemma~\ref{conjugation},
$$H^n=(\lambda_3^{-1}\lambda_4,\lambda_2^{-1}\lambda_4,\lambda_1\lambda_3^{-1}\lambda_4,\lambda_4,\lambda_4\lambda_5^{-1}\lambda_6,\lambda_6,\lambda_7).$$
Applying this to $H_1$, we find that $H_1^n=(-1,-\alpha^{-1},1,-1,-(-\alpha)^{-1},1,\alpha)=H_1$.
Then $H_1^\sigma=(-1,\alpha^q,1,-1,(-\alpha)^q,1,\alpha^q)=(1,\alpha^{q-1},1,1,\alpha^{q-1},1,\alpha^{q-1})H_1=H_1$ and hence $H_1\in\overline{T}_{\sigma n}$. So $N_2$, $N_3$, and $N_4$ belong to $\overline{N}_{\sigma n}$.

Calculations in MAGMA show that $N_1^6=1$. Now we verify that $[N_1,N_2]=1$, $[N_1,N_3]=1$, and $[N_1,N_4]=1$.
Using MAGMA, we see that $[N_1,n_7]=[N_1,h_6n_5]=[N_1,h_1h_4n_{58}n_{59}]=1$. By Lemma~\ref{conjugation}, 
$$H^{-1}H^{N_1}=(\lambda_3^{-1},\lambda_2^{-2},\lambda_1\lambda_3^{-1}\lambda_4^{-1},\lambda_4^{-2},
\lambda_5^{-2},\lambda_6^{-2},\lambda_7^{-2}).$$
Therefore,  $H_1^{-1}H_1^{N_1}=h_2h_5h_7=1$ and hence $N_1$ commutes with $N_2$, $N_3$ and $N_4$ by Lemma~\ref{commutator}.
So $K\simeq\langle N_1\rangle\times\langle N_2,N_3,N_4\rangle$ 

Now we prove that $N_2^2=1$ and $[N_2,N_3]=[N_2,N_4]=1$. Since $n_7^2=h_7$, we have $N_2^2=H_1H_1^{n_7}h_7$. By Lemma~\ref{conjugation},
$H^{n_7}=(\lambda_1,\lambda_2,\lambda_3,\lambda_4,\lambda_5,\lambda_6,\lambda_6\lambda_7^{-1})$ and so $H_1^{n_7}=h_7H_1$. Therefore, $N_2^2=h_2h_5h_7h_7^2=1$.
Using MAGMA, we see that $[n_7,h_6n_5]=h_7$ and $[n_7,h_1h_4n_{58}n_{59}]=1$.  By Lemma~\ref{conjugation},
$$H^{n_5}=(\lambda_1,\lambda_2,\lambda_3,\lambda_4,\lambda_4\lambda_5^{-1}\lambda_6,\lambda_6,\lambda_7),$$
$$H^{n_{58}n_{59}}=(\lambda_1\lambda_6^{-1},\lambda_5\lambda_6^{-2},\lambda_3\lambda_6^{-2},\lambda_4\lambda_6^{-3},
\lambda_2\lambda_6^{-2},\lambda_6^{-1},\lambda_6^{-1}\lambda_7).$$
Applying to $H_1$, we find that $H_1^{n_5}=H_1$ and $H_1^{n_{58}n_{59}}=h_2h_5H_1=h_7H_1$. 
We see above that $H_1^{n_7}=H_1h_7$, so $h_7H_1^{n_5}=H_1^{n_7}=H_1^{n_{58}n_{59}}$.
It follows from Lemma~\ref{commutator} that $[N_2,N_3]=[N_2,N_4]=1$.

Finally, we show that $\langle N_3, N_4 \rangle\simeq D_8$. Namely, we verify that $N_3^2=N_4^2=(N_3N_4)^4=1$.
Using MAGMA, we see that $(h_6n_5)^2=1$ and $(h_1h_4n_{58}n_{59})^2=h_7$.
By Lemma~\ref{conjugation},
$$(Hh_6n_5)^2=(\lambda_1^2,\lambda_2^2,\lambda_3^2,\lambda_4^2,\lambda_4\lambda_6,\lambda_6^2,\lambda_7^2),$$
$$(Hh_1h_4n_{58}n_{59})^2=(\lambda_1^2\lambda_6^{-1},\lambda_2\lambda_5\lambda_6^{-2},\lambda_3^2\lambda_6^{-2},\lambda_4^2\lambda_6^{-3},\lambda_2\lambda_5\lambda_6^{-2},1,-\lambda_6^{-1}\lambda_7^2).$$
Applying these equations to $H_1$, we find that $N_3^2=N_4^2=1$.
Since 
\begin{multline*}
(H_1h_6n_5)(H_1h_1h_4n_{58}n_{59})=(H_1H_1^{n_5})(h_6n_5h_1h_4n_{58}n_{59})=(H_1^2)(h_6n_5h_1h_4n_{58}n_{59})=\\=h_6n_5H_1h_1h_4n_{58}n_{59}, 
\end{multline*}
we infer that $(N_3N_4)^4=(h_6n_5h_1h_4n_{58}n_{59})^4=1$. 
Thus, $K\simeq \mathbb{Z}_6\times \mathbb{Z}_2\times D_8\simeq C_W(w)$
and hence $K$ is a required complement.

\textbf{Torus 13.} In this case $w=w_3w_2w_5w_4$ and $C_W(w)\simeq \mathbb{Z}_6\times \mathbb{Z}_2\times S_4$. Using GAP, we see that $$C_W(w)\simeq\langle a,b,c,d~|~a^6=b^2=c^2=d^2=[a,b]=[a,c]=[a,d]=(bd)^2=(bc)^4=(cd)^3=1\rangle$$
and $w$, $w_7$, $w_{23}w_{24}$,  $w_{20}w_{21}$ are elements of $C_W(w)$ that satisfy this set of relations.

Let $\alpha,\beta\in\overline{\F}_p$ such that $\alpha^2=-1$ and $\beta^{q-1}=(-1)^{(q+1)/2}$. Put $n=n_3n_2n_5n_4$, $N_1=H_1n$, $N_2=H_2n_7$, $N_3=h_4h_6n_{23}n_{24}$, and $N_4=h_3h_5n_{20}n_{21}$, where $H_1=(-1,\alpha,1,-1,-\alpha,1,\alpha)$ and $H_2=(1,\alpha,-1,1,\alpha,1,\beta)$.

We claim that $K=\langle N_1, N_2, N_3, N_4 \rangle$ is a complement for $\overline{T}_{\sigma n}$ in $\overline{N}_{\sigma n}$.

First, we verify that $H_1$ and $H_2$ belong to $\overline{T}_{\sigma n}$. By Lemma~\ref{conjugation},
$$H^n=(\lambda_1, \lambda_3\lambda_4^{-1}\lambda_5, \lambda_1\lambda_2\lambda_4^{-1}\lambda_5,
\lambda_2\lambda_3\lambda_4^{-1}\lambda_5,\lambda_2\lambda_3\lambda_4^{-1}\lambda_6,
\lambda_6,\lambda_7).$$
Therefore,  $H_1^n=H_1$ and $H_2^n=(1,-\alpha,-1,1,-\alpha,1,\beta)$.
Then $$H_1^{n\sigma}=H_1(1,(-1)^{(q-1)/2},1,1,(-1)^{(q-1)/2},1,(-1)^{(q-1)/2})=H_1$$
and $H_2^{n\sigma}=(1,\alpha(-1)^{(q+1)/2},-1,1,\alpha(-1)^{(q+1)/2},1,\beta^q)$.
Since $\beta^q=\beta(-1)^{(q+1)/2}$, we infer that $H_2^{n\sigma}=H_2$ and hence $H_1$, $H_2$
belong to $\overline{T}_{\sigma n}$.

Calculations in MAGMA show that $[n,n_7]=[n,N_3]=[n,N_4]=1$, and hence $N_1$, $N_2$, $N_3$, and $N_4$ lie in~$\overline{N}_{\sigma n}$. Since $n^6=1$, Lemma~\ref{conjugation} implies that
$(Hn)^6=(\lambda_1^6,\lambda_1^3\lambda_6^3,\lambda_1^6\lambda_6^3,\lambda_1^6\lambda_6^6,
\lambda_1^3\lambda_6^6,\lambda_6^6,\lambda_7^6).$ So $N_1^6=(1,-1,1,1,-1,1,-1)=1$.
Now we prove that $N_1$ commutes with $N_2$, $N_3$, and $N_4$. By Lemma~\ref{commutator}, it suffices to verify that $H_1^{-1}H_1^{N_3}=H_1^{-1}H_1^{N_4}=1$ and $H_1^{-1}H_1^{n_7}=H_2^{-1}H_2^n$. By Lemma~\ref{conjugation}, we see that
$$H^{-1}H^{N_3}=(1,\lambda_2^{-1}\lambda_3\lambda_6^{-1}\lambda_7,\lambda_1\lambda_2\lambda_3^{-1}\lambda_6^{-1}\lambda_7,\lambda_1\lambda_6^{-2}\lambda_7^2,\lambda_1\lambda_6^{-2}\lambda_7^2,\lambda_1\lambda_6^{-2}\lambda_7^2,1),$$
$$H^{-1}H^{N_4}=(\lambda_1^{-2}\lambda_6,\lambda_1^{-1}\lambda_2^{-1}\lambda_5,\lambda_1^{-2}\lambda_6,\lambda_1^{-2}\lambda_6,\lambda_1^{-1}\lambda_2\lambda_5^{-1}\lambda_6,1,1),$$
$$H^{-1}H^{n_7}=(1,1,1,1,1,1,\lambda_6\lambda_7^{-2}).$$
Therefore, $H_1^{-1}H_1^{N_3}=1$, $H_1^{-1}H_1^{N_4}=1$, and $H_1^{-1}H_1^{n_7}=h_7$.
We know that $H_2^n=(1,-\alpha,-1,1,-\alpha,1,\beta)=h_2h_5H_2=h_7H_2$ and hence $H_2^{-1}H_2^n=H_1^{-1}H_1^{n_7}$.
Thus, $K\simeq \langle N_1\rangle\times \langle N_2, N_3, N_4\rangle$.

Calculations in MAGMA show that $N_3^2=N_4^2=(N_3N_4)^3=1$, so it remains to verify that $N_2^2=(N_2N_3)^4=(N_2N_4)^2=1$.
Since $n_7^2=h_7$, $(n_7N_3)^4=h_2h_3$, and $(n_7N_4)^2=h_7$, Lemma~\ref{conjugation} implies that
$$ (Hn_7)^2=(\lambda_1^2,\lambda_2^2,\lambda_3^2,\lambda_4^2,\lambda_5^2,\lambda_6^2,-\lambda_6),$$
$$ (Hn_7N_3)^4=(\lambda_1^4,-\lambda_1\lambda_2^2\lambda_3^2\lambda_6^{-2},-\lambda_1^3\lambda_2^2\lambda_3^2\lambda_6^{-2},
\lambda_1^4\lambda_4^4\lambda_6^{-4},\lambda_1^4\lambda_5^4\lambda_6^{-4},\lambda_1^4,\lambda_1^2),$$
$$ (Hn_7N_4)^2=(\lambda_6,\lambda_1^{-1}\lambda_2\lambda_5,\lambda_1^{-2}\lambda_3^2\lambda_6,\lambda_1^{-2}\lambda_4^2\lambda_6,\lambda_1^{-1}\lambda_2\lambda_5\lambda_6,\lambda_6^2,-\lambda_6).$$
We apply these equations to $H_2$ and obtain $N_2^2=h_2h_5h_7=1$,  $(N_2N_3)^4=1$, and $(N_2N_4)^2=h_2h_5h_7=1$.
Thus, $K\simeq C_W(w)$ and hence $K$ is a complement for $\overline{T}_{\sigma n}$.

\textbf{Torus 19.} In this case $w=w_2w_5w_3w_4w_6$ and $C_W(w)=\langle w\rangle\times\langle w_0\rangle\times\langle w_{63} \rangle\simeq \mathbb{Z}_{8}\times \mathbb{Z}_2\times \mathbb{Z}_2$.

Put $n=n_2n_5n_3n_4n_6$. Using MAGMA, we see that $[n, n_{63}]=1$ and hence $n_{63}\in \overline{N}_{\sigma n}$.

Let $\alpha$ be an element of $\overline{\F}_p$ such that $\alpha^2=-1$. 
Put $H_2=(-1,\alpha,1,-1,-\alpha,1,\alpha)$. 
By Lemma~\ref{conjugation},
$$H^{n}=(\lambda_1, \lambda_3\lambda_4^{-1}\lambda_5, \lambda_1\lambda_2\lambda_4^{-1}\lambda_5, \lambda_2\lambda_3\lambda_4^{-1}\lambda_5, \lambda_2\lambda_3\lambda_4^{-1}\lambda_5\lambda_6^{-1}\lambda_7, \lambda_5\lambda_6^{-1}\lambda_7, \lambda_7).$$
Using this equality, we see that $H_2^{\sigma n}=H_2^{\sigma}=H_2$ and hence $H_2\in\overline{T}_{\sigma n}$.

Let $N_2=H_2n_{63}$. We claim that $K=\langle n, n_0, N_1 \rangle$ is a complement for $\overline{T}_{\sigma n}$ in $\overline{N}_{\sigma n}$.
Calculations in MAGMA show that $n^8=1$.
Since $H_2^n=H_2$, we have $H_2^{-1}H_2^n=1$ and the equality $nN_2=N_2n$ follows from Lemma~\ref{commutator}. 
So $K\simeq\langle n \rangle\times\langle n_0, N_2\rangle$.

By Lemma~\ref{conjugation}, $H^{n_{63}}=(\lambda_1^{-1}, \lambda_1^{-2}\lambda_2, \lambda_1^{-3}\lambda_3, \lambda_1^{-4}\lambda_4, \lambda_1^{-3}\lambda_5, \lambda_1^{-2}\lambda_6, \lambda_1^{-1}\lambda_7).$
Then $N_2^2=H_2H_2^{n_{63}}h_2h_3=1$. Finally, we see that $H_2^{-1}H_2^{n_0}=H_2^{-2}=(1,-1,1,1,-1,1,-1)=1$ and hence $N_2N_0=N_0N_2$ by Lemma~\ref{commutator}.
Thus, $K=\langle N_1 \rangle\times\langle N_2 \rangle\times\langle N_0 \rangle\simeq \mathbb{Z}_{8}\times \mathbb{Z}_2\times \mathbb{Z}_2$, as claimed.

\textbf{Torus 20.} In this case $w=w_{23}w_5w_4w_3w_2$ and $C_W(w)=\langle w\rangle\times\langle w_0\rangle\times\langle w_{63} \rangle\simeq \mathbb{Z}_{12}\times \mathbb{Z}_2\times \mathbb{Z}_2$.

Let $\alpha$ be an element of $\overline{\F}_p$ such that $\alpha^2=-1$ and $n=n_{23}n_5n_4n_3n_2$.
Put $N_1=h_1n$ and $N_2=H_2n_{63}$, where $H_2=(-1,-\alpha,1,1,\alpha,-1,\alpha)$. 
By Lemma~\ref{conjugation},
$$H^{n}=(\lambda_1, \lambda_1\lambda_3^{-1}\lambda_4\lambda_6^{-1}\lambda_7, \lambda_1\lambda_3^{-1}\lambda_4, \lambda_1^2\lambda_3^{-2}\lambda_4\lambda_5\lambda_6^{-1}\lambda_7, \lambda_1^2\lambda_3^{-2}\lambda_4\lambda_7, \lambda_1\lambda_2\lambda_3^{-1}\lambda_7, \lambda_7).$$
Using this equality, we see that 
$H_2^{n}=h_2h_3H_2$ and hence $H_2^{n\sigma}h_2h_3=H_2$. Calculations in MAGMA show that $[N_1,n_{63}]=h_2h_3$.
Therefore,  $H_2^{n\sigma}[N_1,n_{63}]=H_2$ and $N_2\in \overline{N}_{\sigma n}$ by Lemma~\ref{normalizer}.

We claim that $K=\langle N_1, N_2, n_0 \rangle$ is a complement for $\overline{T}_{\sigma n}$ in $\overline{N}_{\sigma n}$.
Calculations in MAGMA show that $N_1^{12}=1$. By Lemma~\ref{conjugation},
$$H^{n_{63}}=(\lambda_1^{-1}, \lambda_1^{-2}\lambda_2, \lambda_1^{-3}\lambda_3, \lambda_1^{-4}\lambda_4, \lambda_1^{-3}\lambda_5, \lambda_1^{-2}\lambda_6, \lambda_1^{-1}\lambda_7).$$
Then $N_2^2=H_2H_2^{n_{63}}h_2h_3=1$. 
Furthermore, we have  $H_2^{-1}H_2^{N_0}=H_2^{-2}=(1,-1,1,1,-1,1,-1)=1$ and hence $N_2n_0=n_0N_2$ by Lemma~\ref{commutator}.
Thus, $K\simeq \mathbb{Z}_{12}\times \mathbb{Z}_2\times \mathbb{Z}_2$, as claimed.

\textbf{Torus 26.} In this case $w=w_1w_4w_6w_3w_7$ and $C_W(w)=\langle w\rangle\times\langle w_0\rangle\times\langle w_{59} \rangle\simeq \mathbb{Z}_{12}\times \mathbb{Z}_2\times \mathbb{Z}_2$.

Let $\alpha$ be an element of $\overline{\F}_p$ such that $\alpha^2=-1$.
Put $n=n_1n_4n_6n_3n_7$, $N_1=h_2n$, and $N_2=H_2h_3n_{59}$, where $H_2=(1,\alpha,1,1,-\alpha,-1,\alpha)$. 
Calculations in MAGMA show that $[N_1,h_3n_{59}]=1$ and hence $h_3n_{59}\in \overline{N}_{\sigma n}$.
By Lemma~\ref{conjugation},
$$H^{n}=(\lambda_3^{-1}\lambda_4, \lambda_2, \lambda_1\lambda_3^{-1}\lambda_4, \lambda_1\lambda_2\lambda_3^{-1}\lambda_5, \lambda_5, \lambda_5\lambda_7^{-1}, \lambda_6\lambda_7^{-1}).$$
Using this equality, we see that 
$H_2^{\sigma n}=H_2^{\sigma}=H_2$. Thus, $H_2\in\overline{N}_{\sigma n}$ and $N_2\in\overline{N}_{\sigma n}$.

We claim that $K=\langle N_1, n_0, N_2\rangle$ is a complement for $\overline{T}_{\sigma n}$ in $\overline{N}_{\sigma n}$.
Calculations in MAGMA show that $N_1^{12}=1$.
Since $H_2^n=H_2$, we have $H_2^{-1}H_2^n=1$. Hence, the equality $N_1N_2=N_2N_1$ follows from Lemma~\ref{commutator}. 

By Lemma~\ref{conjugation},
$$H^{n_{59}}=(\lambda_1\lambda_2\lambda_5^{-1}, \lambda_2^2\lambda_5^{-1}, \lambda_2^2\lambda_3\lambda_5^{-2}, \lambda_2^3\lambda_4\lambda_5^{-3}, \lambda_2^3\lambda_5^{-2}, \lambda_2^2\lambda_5^{-2}\lambda_6, \lambda_2\lambda_5^{-1}\lambda_7).$$
Then $N_2^2=H_2H_2^{n_{59}}h_1h_4=1$. 
Furthermore, we see that $H_2^{-1}H_2^{n_0}=H_2^{-2}=(1,-1,1,1,-1,1,-1)=1$ and by Lemma~\ref{commutator} we have $N_2n_0=n_0N_2$.
Thus, we conclude that $K=\langle N_1 \rangle\times\langle N_2 \rangle\times\langle n_0 \rangle\simeq \mathbb{Z}_{12}\times \mathbb{Z}_2\times \mathbb{Z}_2$, as claimed.

\textbf{Torus 30.} In this case $w=w_{39}w_{3}w_{5}w_1w_4w_6$ and $C_W(w)=\langle w\rangle\times\langle w_0\rangle\times\langle w_{53}\rangle\simeq \mathbb{Z}_8\times \mathbb{Z}_2\times \mathbb{Z}_2$. 
Put $n=n_{39}n_3n_5n_1n_4n_6$. Using MAGMA, we see that $[n, n_{53}]=1$ and hence $n_{53}\in \overline{N}_{\sigma n}$.

Let $\alpha$ be an element of $\overline{\F}_p$ such that $\alpha^2=-1$. 
Put $H_2=(1,\alpha,-1,-1,-\alpha,1,-\alpha)$. 
By Lemma~\ref{conjugation},
\begin{multline*}
H^{n}=(\lambda_3\lambda_4^{-1}\lambda_5\lambda_7^{-1}, \lambda_1\lambda_2\lambda_4^{-1}\lambda_5\lambda_7^{-1}, \lambda_2\lambda_3\lambda_4^{-2}\lambda_5^2\lambda_7^{-1}, \lambda_1\lambda_2\lambda_3\lambda_4^{-2}\lambda_5^2\lambda_7^{-1}, \lambda_1\lambda_2\lambda_3\lambda_4^{-2}\lambda_5^2\lambda_6^{-1},\\ \lambda_1\lambda_4^{-1}\lambda_5^2\lambda_6^{-1}, \lambda_1\lambda_4^{-1}\lambda_5).
\end{multline*}

Using this equality, we see that $H_2^{\sigma n}=(1,-\alpha,-1,-1,\alpha,1,\alpha)^{\sigma}=H_2$ and hence $H_2\in \overline{T}_{\sigma n}$.

Let $N_1=n$ and $N_2=H_2n_{53}$. We claim that $K=\langle N_1, n_0, N_2 \rangle$ is a complement for $\overline{T}_{\sigma n}$ in $\overline{N}_{\sigma n}$.
Calculations in MAGMA show that $N_1^8=1$.
Since $H_2^n=H_2$, we have $H_2^{-1}H_2^n=1$ and the equality $N_1N_2=N_2N_1$ follows from Lemma~\ref{commutator}. By Lemma~\ref{conjugation},
$$H^{n_{53}}=(\lambda_1\lambda_2^{-1}\lambda_7, \lambda_2^{-1}\lambda_7^2, \lambda_2^{-2}\lambda_3\lambda_7^2, \lambda_2^{-3}\lambda_4\lambda_7^3, \lambda_2^{-2}\lambda_5\lambda_7^2, \lambda_2^{-1}\lambda_6\lambda_7, \lambda_7).$$
Then $N_2^2=H_2H_2^{n_{53}}h_1h_4h_6=(1,\alpha,-1,-1,-\alpha,1,-\alpha)(-1,\alpha,-1,1,-\alpha,-1,-\alpha)h_1h_4h_6=1$. 
Further,  $H_2^{-1}H_2^{n_0}=H_2^{-2}=(1,-1,1,1,-1,1,-1)=1$ and by Lemma~\ref{commutator} we have $N_2n_0=n_0N_2$.
Thus, $K=\langle N_1 \rangle\times\langle N_2 \rangle\times\langle n_0 \rangle\simeq \mathbb{Z}_{8}\times \mathbb{Z}_2\times \mathbb{Z}_2$, as claimed. 
\end{proof}

\begin{lemma} Suppose that $w$ or $ww_0$ is conjugate to one of the following:
$w_1w_5w_3w_6$, $w_1w_4w_3w_2$, $w_1w_5w_3w_6w_2$, $w_1w_4w_5w_3w_{53}$, $w_1w_4w_6w_3w_5$,
 $w_1w_5w_2w_3w_6w_{53}$, $w_1w_4w_6w_3w_5w_{53}$, $w_1w_4w_6w_3w_2w_5$,
$w_1w_4w_6w_{16}w_3w_2w_6$, $w_1w_4w_{16}w_3w_2w_{40}$, $w_1w_4w_6w_2w_3w_7$, or
$w_1w_4w_6w_3w_5w_7$. Then $T$ has a complement in its algebraic normalizer.
\end{lemma} 

\begin{proof} Our strategy is the same in all cases.
For an element $w$, we choose $n\in\mathcal{T}$ such that $\pi(n)=w$ and a set of relations that defines $C_W(w)$. Then we verify that
generators of a subgroup of $\overline{N}_{\sigma n}$ satisfy this set of relations and generate a complement for $\overline{T}_{\sigma n}$. All data is listed in Table~\ref{t:split_E7}.
As an example, we consider $w=w_1w_5w_3w_6$ that corresponds to Torus 10 in Table~\ref{t:main:E7}.
In this case $C_W(w)=\langle ww_0 \rangle \times \langle w_2, w_{53} \rangle\times\langle w_2w_{32}w_{35}w_{46}, w_2w_{28}w_{42}w_{43} \rangle\simeq\mathbb{Z}_6\times S_3\times S_3$. 
Using GAP, we see that $C_W(w)$ has the following presentation:
\begin{multline*}
\langle a,b,c,d,e~|~a^6=b^2=c^2=d^2=e^2=[a,b]=[a,c]=[a,d]=[a,e]=\\(bc)^3=(de)^3=[b,d]=[b,e]=[c,d]=[c,e]=1\rangle.
\end{multline*}
Moreover, we verify that  $ww_0$, $w_2$, $w_{53}$, $w_2w_{32}w_{35}w_{46}$, $w_2w_{28}w_{42}w_{43}$ satisfy this set of relations.
Consider $a=n_0n$, $b=h_{53}n_2$, $c=h_2n_{53}$, $d=h_1h_6n_2n_{32}n_{35}n_{46}$, and $e=h_1h_3h_6n_2n_{28}n_{42}n_{43}$.
Then calculations in MAGMA show that $a$, $b$, $c$, $d$, $e$ satisfy relations for $C_W(w)$ and hence $K=\langle a,b,c,d,e \rangle$
is a homomorphic image of $C_W(w)$. On the other hand, we have 
$\pi(K)=C_W(w)$ and hence $K\simeq C_W(w)$. Let $n=n_1n_5n_3n_6$. Finally, we verify that
$[n,a]=[n,b]=[n,c]=[n,d]=[n,e]=1$ and hence $a,b,c,d,e\in\overline{N}_{\sigma n}$ by Lemma~\ref{normalizer}. Thus, $K$ is a complement for $\overline{T}_{\sigma n}$ in $\overline{N}_{\sigma n}$.

Other cases can be verified in the same way. The first column of Table~\ref{t:split_E7} contains numbers of tori in accordance with Table~\ref{t:main:E7}. The second column for each $w$ contains a set of relations $S(w)$ that defines $C_W(w)$. 
The third column contains examples of generators of a complement. All such generators lie in $\mathcal{T}$.
Therefore,  it is easy to verify in MAGMA that generators satisfy $S(w)$. The natural preimage of $w$ in $\mathcal{T}$
is denote by $n$. In each case we choose an element $x\in\mathcal{T}$ and consider the groups $\overline{T}_{\sigma x}$ and $\overline{N}_{\sigma x}$. Usually $x=n$ but sometimes they differ. To verify that a generator $y$ lies in $\overline{N}_{\sigma x}$, one can check that $[x,y]=1$ and apply Lemma~\ref{normalizer}. Elements from the third column
generate a complement for $\overline{T}_{\sigma x}$ and hence $N(G,T)$ splits over $T$
by Remark~\ref{r:nonsplit}.

For convenience, we add all verified equations in~\cite{StarE7}.

\begin{longtable}{|c|l|l|}
\caption{Splitting normalizers of maximal tori of ~$E^{ad}_7(q)$}\label{t:split_E7} \\
\hline
 \No  & Defining relations for $C_W(w)$ & Generators for a complement  \\ 
      &  &  of $\overline{T}_{\sigma x}$ in $\overline{N}_{\sigma x}$ \\ 
\hline\endfirsthead
\multicolumn{3}{c}{\textit{Table~\ref{t:split_E7} (continued)}} \\
\hline
 \No  & Defining relations  & Generators for a complement  \\ 
      & for $C_W(w)$ &  of $\overline{T}_{\sigma x}$ in $\overline{N}_{\sigma x}$ \\ 
\hline\endhead
\hline
\endfoot
\endlastfoot
  10  & $a^6=b^2=c^2=d^2=e^2=1$,  & $x=n$, $a=n_0n$,     \\
      & $[a,b]=[a,c]=[a,d]=[a,e]=1$, & $b=h_{53}n_2$, $c=h_2n_{53}$,  \\
      &  $(bc)^3=(de)^3=[b,d]=[b,e]=1$,  &  $d=h_1h_6n_2n_{32}n_{35}n_{46}$,   \\
      & $[c,d]=[c,e]=1$ & $e=h_1h_3h_6n_2n_{28}n_{42}n_{43}$ \\ \hline
  12  & $a^{10}=[a,b]=[a,c]=1$, & $x=n$, $a=n_0n$, \\ 
      & $b^2=c^2=(bc)^3=1$ & $b=h_7n_6$, $c=h_6n_7$ \\ \hline
  
  15 & $a^6=b^2=c^2=d^2=(cd)^3=1$, &  $a=x=h_4n$, $b=n_0$, \\
     &  $[a,b]=[a,c]=[a,d]=1,$  & $c=h_4n_{32}n_{35}n_{46}$, \\ 
     &  $[b,c]=[b,d]=1$ & $d=h_1h_4h_6n_{28}n_{42}n_{43}$ \\ \hline
  17 & $a^{10}=b^2=[a,b]=1$  & $a=x=h_2n$, $b=n_0$ \\ \hline
  18 & $a^6=b^2=c^2=1,$ & $a=x=h_2n$, $b=n_0$, \\
     & $[a,b]=[a,c]=[b,c]=1$ & $c=h_3h_7n_{53}$ \\ \hline
  21 & $a^3=b^2=c^3=d^4=1$, & $a=x=n$, $b=n_0$, $c=n_5n_6$, \\
     & $[a,b]=[a,c]=[a,d]=1$, &  $d=h_1h_3h_4n_1n_4n_8n_{22}$  \\
     & $[b,c]=[b,d]=1$, &  \\
     & $c^{-1}a(d^{-1}c^{-1})^2d^{-1}=1$, &  \\
     & $da^{-1}c^{-1}d(dc^{-1})^2d^{-1}cd^{-1}c^{-1}=1$  & \\ \hline
  22 & The same as for Torus 15 & $a=x=h_1n$, \\
     &  & $b=n_0$, $c=h_{33}n_{53}$,  \\
     &  & $d=h_{53}n_{33}$ \\ \hline
  23 & $a^{12}=b^2=[a,b]=1$ & $a=x=n$, $b=n_0$ \\ \hline
  24 & $a^9=b^2=[a,b]=1$  & $a=x=n$, $b=n_0$ \\ \hline
  25 & $a^6=[a,b]=[a,c]=[a,d]=1$, & $x=n$, $a=x^2n_0$,  \\
     & $d^3=c^{-1}bcb=b^2c^{-2}=1,$ & $b=h_1h_2h_5n_3n_6n_{22}n_{32}$, \\
     & $bdc^{-1}d^{-1}=bcdbd^{-1}=1$  & $c=h_2h_3h_4h_5n_3n_6n_{16}n_{38}$, \\
     &  & $d=h_1h_2h_4h_6n_1n_4n_6n_{15}n_{23}n_{46}$  \\ \hline
  27 & $a^{30}=1$ & $x=n$, $a=xn_0$ \\ \hline
  29 & $a^{14}=1$ & $x=n$, $a=xn_0$ \\ \hline
\end{longtable}
\end{proof}

\begin{longtable}{|l|l|c|l|c|}
\caption{Splitting of the normalizers of maximal tori in $E^{ad}_7(q)$}\label{t:main:E7} \\
\hline
 \No  & Representative $w$ & $|w|$ & Cyclic structure of $(2,q-1).T$ & \\ 
\hline\endfirsthead
\multicolumn{5}{c}{\textit{Table~\ref{t:main:E7} (continued)}} \\
\hline
\No  & Representative $w$ & $|w|$ & Cyclic structure of $(2,q-1).T$ & \\
\hline\endhead
\endfoot
\endlastfoot
  1  & $1$ & 1 &   $(q-1)^7$ &  -- \\ \hline
  2  & $w_1$ & 2  & $(q-1)^5\times(q^2-1)$ &  -- \\ \hline
  3  & $w_1w_2$ & 2 &   $(q-1)^3\times(q^2-1)^2$ & -- \\ \hline
  4  & $w_3w_1$ & 3 &   $(q-1)^4\times(q^3-1)$ & + \\ \hline
  5  & $w_2w_3w_5$ & 2 &  $(q-1)\times(q^2-1)^3$ & -- \\ \hline
  6  & $w_1w_3w_5$ & 6 &   $(q-1)^2\times(q^2-1)\times(q^3-1)$ & + \\ \hline
  7 & $w_1w_3w_4$ & 4 &   $(q-1)^3\times(q^4-1)$ & -- \\ \hline
  8 & $w_1w_4w_6w_{53}$ & 2 &  $(q-1)\times(q+1)^2\times(q^2-1)^2$ & -- \\ \hline
  9 & $w_1w_2w_3w_5$ & 6 &   $(q-1)\times(q^2-1)\times(q+1)(q^3-1)$& + \\ \hline
  10 & $w_1w_5w_3w_6$ & 3 &   $(q-1)\times(q^3-1)^2$& + \\  \hline
  11 & $w_1w_4w_6w_3$ & 4 &   $(q-1)\times(q^2-1)\times(q^4-1)$ & -- \\ \hline
  12 & $w_1w_4w_3w_2$ & 5 &   $(q-1)^2\times(q^5-1)$ & + \\ \hline
  13 & $w_3w_2w_5w_4$ & 6 &   $(q-1)\times(q^2-1)\times(q-1)(q^3+1)$ & + \\ \hline
  14 & $w_3w_2w_4w_{16}$ & 4 &  $(q-1)\times((q-1)(q^2+1))^2$ & -- \\ \hline
  15 & $w_1w_5w_3w_6w_2$ & 6 &  $(q^3-1)\times(q+1)(q^3-1)$ & + \\ \hline
  16 & $w_1w_4w_6w_3w_{53}$ & 4 &   $(q-1)\times(q+1)^2\times(q^4-1)$ & -- \\ \hline
  17 & $w_1w_4w_5w_3w_{53}$ &  10 &   $(q-1)\times(q+1)(q^5-1)$ & + \\ \hline
  18 & $w_1w_4w_6w_3w_5$ & 6  &   $(q-1)\times(q^6-1)$ & + \\ \hline
  19 & $w_2w_5w_3w_4w_6$ & 8  &   $(q-1)\times(q^2-1)(q^4+1)$ & + \\ \hline
  20 & $w_{23}w_5w_4w_3w_2$ & 12 &  $(q-1)\times(q-1)(q^2+1)(q^3+1)$ & + \\ \hline
  21 & $w_1w_5w_2w_3w_6w_{53}$ & 3 &  $(q^2+q+1)^2\times(q^3-1)$ & + \\ \hline
  22 & $w_1w_4w_6w_3w_5w_{53}$ & 6 &   $(q^3+1)\times(q^3-1)\times(q+1)$& + \\ \hline
  23 & $w_1w_4w_6w_3w_2w_5$ & 12 &  $(q^3-1)(q^4-q^2+1)$ & + \\ \hline
  24 & $w_1w_4w_{16}w_3w_2w_6$ & 9 &  $(q-1)(q^6+q^3+1)$ & + \\ \hline
  25 & $w_1w_4w_{16}w_3w_2w_{40}$ & 6 &  $(q^2-q+1)\times(q-1)(q^4+q^2+1)$& +\\ \hline
  26 & $w_1w_4w_{6}w_3w_7$ & 12 &   $(q^3-1)\times(q^4-1)$ & + \\ \hline
  27 & $w_1w_4w_{6}w_2w_3w_7$ & 15 &   $(q^5-1)(q^2+q+1)$ & + \\ \hline
  28 & $w_3w_2w_4w_{16}w_7$ & 4 &  $(q-1)(q^2+1)\times(q^2-1)\times(q^2+1)$ & -- \\ \hline
  29 & $w_1w_4w_{6}w_3w_5w_7$ & 7 &   $q^7-1$ & + \\ \hline
  30 & $w_{39}w_3w_5w_1w_4w_6$ & 8 &  $(q^4+1)\times(q-1)(q^2+1)$ & + \\
  \hline
\end{longtable}

\section{The proof for type $E_8$}

In this section, we prove Theorem~\ref{th:E8}. We assume that $G$ is a finite group of Lie type $E_8$ and $q$ is odd. 
The extended Dynkin diagram of $E_8$ is the following.
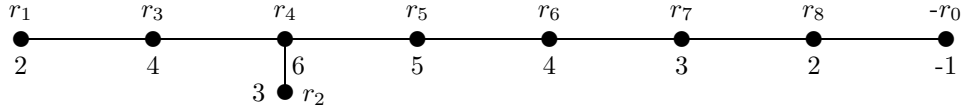
\vspace{1em}

\begin{picture}(330,60)(-50,-30)
\put(0,0){\line(1,0){50}} \put(50,0){\line(1,0){50}}
\put(100,0){\line(1,0){50}} \put(150,0){\line(1,0){50}}
\put(200,0){\line(1,0){50}} \put(250,0){\line(1,0){50}}
\put(0,0){\circle*{6}} \put(50,0){\circle*{6}}
\put(100,0){\circle*{6}} \put(150,0){\circle*{6}}
\put(200,0){\circle*{6}} \put(250,0){\circle*{6}}
\put(300,0){\circle*{6}} \put(100,0){\line(0,-1){20}}
\put(100,-20){\circle*{6}} \put(0,10){\makebox(0,0){$r_1$}}
\put(0,-10){\makebox(0,0){2}}\put(50,10){\makebox(0,0){$r_3$}}
\put(50,-10){\makebox(0,0){4}}\put(100,10){\makebox(0,0){$r_4$}}
\put(105,-10){\makebox(0,0){6}}\put(150,10){\makebox(0,0){$r_5$}}
\put(150,-10){\makebox(0,0){5}}\put(200,10){\makebox(0,0){$r_6$}}
\put(200,-10){\makebox(0,0){4}}\put(250,10){\makebox(0,0){$r_7$}}
\put(250,-10){\makebox(0,0){3}}\put(300,10){\makebox(0,0){$r_8$}}
\put(300,-10){\makebox(0,0){2}}\put(111,-22){\makebox(0,0){$r_2$}}
\put(90,-20){\makebox(0,0){3}} \put(300,0){\line(1,0){50}}
\put(350,0){\circle*{6}}
\put(350,-10){\makebox(0,0){-1}}\put(350,10){\makebox(0,0){-$r_0$}}
\end{picture}

In this case the Weyl group $W$ is isomorphic to $(2.\Oo_8^+(2)):2$ (in the notation of \cite{Atlas}). So $W$ has the central involution $w_0$. It is easy to verify that $w_0=w_1w_2w_5w_7w_{44}w_{71}w_{89}w_{120}$. 
There are in total 112 conjugacy classes in $W$. 
In contrast to the type $E_7$, it is possible that $w\in W$ and $ww_0$
are conjugate in $W$. There are in total 22 such conjugacy classes.
All the other classes are divided into pairs such that one class in a pair can be obtained from the other by multiplying on $w_0$. This information can be verified in GAP.
We list 67 of representatives of conjugacy classes in Table~\ref{t:main:E8} and the remaining 45 representatives can be obtained by multiplying on $w_0$. 

\begin{remark}
There are two comments concerning~\cite[Table I]{DerF}. The representatives of the conjugacy classes of $W$ in this table were obtained from ~\cite{Car2}. 
\begin{enumerate}[(i)]
\item According to \cite[Table 11]{Car2}, a representative with number 50 has an admissible diagram $D_5\times A_2$. The element $u=w_2w_3w_5w_4w_8w_6w_{120}$ has such a diagram and we choose it for further computations. In~\cite[Table I]{DerF} the representative $u'=w_2w_3w_5w_4w_8w_6w_{35}$ does not have such a diagram, but $u$ is conjugate to $u'w_0$ in $W$. Hence, the representative $u'$ corresponds to a maximal torus with cyclic structure $(q-1)(q^3+1)(q^4+1)$.
\item According to \cite[Table 11]{Car2}, a representative with number 51 has an admissible diagram $D_5(a_1)\times A_2$. In~\cite[Table I]{DerF} the representative $v'=w_{26}w_5w_4w_3w_2w_7w_8$ contains a misprint, because $v'$ has the same admissible diagram as $u'$. The element $v=w_{26}w_5w_4w_3w_2w_{120}w_8$ has a diagram $D_5(a_1)\times A_2$ and we choose it for further computations. Notice that the cyclic structure of the corresponding maximal torus in~\cite[Table I]{DerF} is correct and equal to $(q^2+1)(q^6-1)$.
\end{enumerate}
\end{remark}

Put $n_0=h_2h_5h_7n_1n_2n_5n_7n_{44}n_{71}n_{89}n_{120}$.
Calculations in MAGMA show that $[n_0, n_i]=1$ for $i=1,\ldots,8$ and hence $n_0\in Z(\mathcal{T})$.
We divide our proof into two general cases. First, we consider maximal tori that do not have complements in their algebraic normalizers. If a maximal torus $T$ does not have a complement and corresponds to the conjugacy class of element $w\in W$ then we provide a lift of $w$ of the same order. Secondly, we consider the remaining tori and present generators of their complements (up to conjugation).

\subsection{Type $E_8$: non-splitting cases}
Throughout this subsection we suppose that $T$ is a maximal torus corresponding to the conjugacy class of $w\in W$. 
We write $H=(\lambda_1,\lambda_2,\lambda_3,\lambda_4,\lambda_5,\lambda_6,\lambda_7,\lambda_8)$ for an arbitrary element of $T$. This notation means that $H=\prod\limits_{i=1}^8h_{r_i}(\lambda_i)$.
The main tool to show that $T$ does not have a complement in $N(G,T)$ is the following assertion similar to Lemma~\ref{l:torie7}.
\begin{lemma}\label{l:torie8}
Let $w\in W$ and $C_W(w)\geqslant\langle w_2w_5,w_{61},w_{97}\rangle \simeq\mathbb{Z}_2\times\mathbb{Z}_2\times\mathbb{Z}_2$. Suppose that $n\in\overline{N}$
such that $\pi(n)=w$ and $n_2n_5,n_{61},n_{97}$ lie in $\overline{N}_{\sigma n}$. Then $\overline{T}_{\sigma n}$ does not have a complement in $\overline{N}_{\sigma n}$.
\end{lemma}
\begin{proof}
Assume that $\overline{T}_{\sigma n}$ has a complement $K$ in $\overline{N}_{\sigma n}$.
Let $N_1,N_2,N_3$ be preimages of $w_2w_5$, $w_{61}$, and $w_{97}$ in $K$, respectively. Then $N_1=H_1n_2n_5$, $N_2=H_2n_{49}$, $N_3=H_3n_{63}$, where
$$H_1=(\mu_1,\ldots,\mu_7),
H_2=(\alpha_1,\ldots,\alpha_7),
H_3=(\beta_1,\ldots,\beta_7)$$
are elements of $\overline{T}_{\sigma n}$.

Since $w_{61}^2=1$, it is true that $N_2^2=1$. Computations in MAGMA show
that $n_{61}^2=h_2h_3h_7$. By Lemma~\ref{conjugation}, 
$Hn_{61}^2=(\lambda_1^2, -\lambda_1\lambda_2^2\lambda_6^{-1}\lambda_8,-\lambda_1\lambda_3^2\lambda_6^{-1}\lambda_8,
\lambda_1^2\lambda_4^2\lambda_6^{-2}\lambda_8^2,\lambda_1^2\lambda_5^2\lambda_6^{-2}\lambda_8^2,\lambda_1^2\lambda_8^2,
-\lambda_1\lambda_6^{-1}\lambda_7^2\lambda_8,\lambda_8^2)$.
Therefore, $\alpha_1\alpha_2^2\alpha_8=-\alpha_6$, $\alpha_1^2=\alpha_8^2=1$, and $\alpha_4^2=\alpha_5^2=\alpha_6^2$.

Since $[n_2n_5,n_{61}]=1$, it follows from Lemma~\ref{commutator} that $H_1^{-1}H_1^{n_{61}}=H_2^{-1}H_2^{n_2n_5}$.
By Lemma~\ref{conjugation},
$$H^{-1}H^{n_{61}}=(1,\lambda_1\lambda_6^{-1}\lambda_8,\lambda_1\lambda_6^{-1}\lambda_8,
\lambda_1^2\lambda_6^{-2}\lambda_8^2,\lambda_1^2\lambda_6^{-2}\lambda_8^2,\lambda_1^2\lambda_6^{-2}\lambda_8^2,
\lambda_1\lambda_6^{-1}\lambda_8,1),$$
$$H^{-1}H^{n_2n_5}=(1,\lambda_2^{-2}\lambda_4,1,1,\lambda_4\lambda_5^{-2}\lambda_6,1,1,1).$$
Applying to $H_1$ and $H_2$, we see that 
$$(1,\mu_1\mu_6^{-1}\mu_8,\mu_1\mu_6^{-1}\mu_8,
\mu_1^2\mu_6^{-2}\mu_8^2,\mu_1^2\mu_6^{-2}\mu_8^2,\mu_1^2\mu_6^{-2}\mu_8^2,
\mu_1\mu_6^{-1}\mu_8,1)=(1,\alpha_2^{-2}\alpha_4,1,1,\alpha_4\alpha_5^{-2}\alpha_6,1,1,1).$$ Then 
$\mu_1\mu_6^{-1}\mu_8=1$, so $\alpha_2^2=\alpha_4$ and $\alpha_5^2=\alpha_4\alpha_6$.
Since $\alpha_5^2=\alpha_4^2=\alpha_6^2$, we infer that $\alpha_4=\alpha_6$.
Now the equality $\alpha_1\alpha_2^2\alpha_8=-\alpha_6$ implies that $\alpha_1\alpha_8=-1$.

Since $[n_{97},n_{61}]=1$, it follows from Lemma~\ref{commutator} that $H_3^{-1}H_3^{n_{61}}=H_2^{-1}H_2^{n_{97}}$.
By Lemma~\ref{conjugation},
$$H^{-1}H^{n_{97}}=(\lambda_1^{-2}\lambda_8^2,\lambda_1^{-2}\lambda_8^2,\lambda_1^{-3}\lambda_8^3,\lambda_1^{-4}\lambda_8^4,\lambda_1^{-3}\lambda_8^3,\lambda_1^{-2}\lambda_8^2,\lambda_1^{-1}\lambda_8,1).$$
Therefore,  
\begin{multline*}(\alpha_1^{-2}\alpha_8^2,\alpha_1^{-2}\alpha_8^2,\alpha_1^{-3}\alpha_8^3,\alpha_1^{-4}\alpha_8^4,\alpha_1^{-3}\alpha_8^3,\alpha_1^{-2}\alpha_8^2,\alpha_1^{-1}\alpha_8,1)=\\=(1,\beta_1\beta_6^{-1}\beta_8,\beta_1\beta_6^{-1}\beta_8, \beta_1^2\beta_6^{-2}\beta_8^2,\beta_1^2\beta_6^{-2}\beta_8^2,\beta_1^2\beta_6^{-2}\beta_8^2,
\beta_1\beta_6^{-1}\beta_8,1).
\end{multline*}
Comparing the second and the third coordinates, we find that $\alpha_1^{-2}\alpha_8^2=\alpha_1^{-3}\alpha_8^3$ and hence $\alpha_1^{-1}\alpha_8=1$.
Since $\alpha_1^2=1$, we arrive at a contradiction with $\alpha_1\alpha_8=-1$.
\end{proof}

Now we consider cases where Lemma~\ref{l:torie8} can be applied.

\begin{lemma} Let $w$ or $ww_0$ be an element from the second column of Table~\ref{t:nonsplit_E8}.
If a maximal torus $T$ corresponds to the conjugacy class of $w$ then
$T$ does not have a complement in $N(G,T)$. Nevertheless, $w$ has a lift to $N(G,T)$ of order $|w|$.
\end{lemma}
\begin{proof} By Lemma~\ref{l:ww0}, we can assume that $w$ is an element from the second column of Table~\ref{t:nonsplit_E8}. Using GAP, we find that in each case there exists $w'\in W$ which is conjugate to $w$ 
and such that $w_2w_5$, $w_{61}$, $w_{97}\in C_W(w')$. Examples of such $w'$ for every $w$
are listed in Table~\ref{t:nonsplit_E8}. The first column of this table contains the number of a torus according to Table~\ref{t:main:E8},
and the third column contains an example of $w'$ for $w$. The fourth column contains elements $n_{w'}\in\mathcal{T}$ 
such that $[n_{w'},n_2n_5]=[n_{w'},n_{61}]=[n_{w'},n_{97}]=1$ and $\pi(n_{w'})=w'$.
Then Lemma~\ref{normalizer} yields $n_2n_5$, $n_{61}$, and $n_{97}$ belong to $\overline{N}_{\sigma n_{w'}}$ in each case.
Now Lemma~\ref{l:torie8} implies that $\overline{N}_{\sigma n_{w'}}$ does not split over $\overline{T}_{\sigma n_{w'}}$. By Remark~\ref{r:nonsplit}, $T$ does not have a complement in $N(G,T)$. Finally, the fifth column contains another preimage $n'$ of $w'$ in $\mathcal{T}$ such that $|w'|=|n'|$. The latter equality can be verified in MAGMA. Thus, $n'$ is a required lift of $w'$ to $\overline{N}_{\sigma n'}$. Since $n_0\in Z(\mathcal{T})$, we have $n'n_0$ is a required lift of $w'w_0$ and the lemma is proved.
\end{proof}

\begin{longtable}{|c|l|l|l|}
\caption{Representatives of conjugacy classes satisfying the condition of Lemma~\ref{l:torie8}}\label{t:nonsplit_E8} \\
\hline
 \No & $w$  & $w'\in w^W$ & preimage of $w'$,  \\ 
     &      &             & lift of $w'$ \\ 
\hline\endfirsthead
\multicolumn{4}{c}{\textit{Table~\ref{t:nonsplit_E8} (continued)}} \\
\hline
 \No & $w$  & $w'\in w^W$ & preimage of $w'$,  \\ 
     &      &             & lift of $w'$ \\ 
\hline\endhead
\hline
\endfoot
\endlastfoot
  1  & $1$   &  $1$ &  $1$, $1$ \\ \hline
  2  & $w_1$ & $w_2$ & $n_2$, $h_4n_2$ \\ \hline
  3  & $w_1w_2$ & $w_2w_5$ & $n_2n_5$, $h_4n_2n_5$  \\ \hline
  4  & $w_3w_1$ & $w_3w_{99}$ & $n_3n_{99}$,   $n_3n_{99}$ \\ \hline
  5  & $w_2w_3w_5$ & $w_3w_5w_2$ & $n_3n_5n_2$, $h_4n_3n_5n_2$ \\ \hline
  6  & $w_1w_3w_5$ &  $w_{61}w_{99}w_3$ & $n_{61}n_{99}n_3$, $h_6n_{61}n_{99}n_3$ \\ \hline
  7  & $w_1w_3w_4$ & $w_2w_4w_{18}$ & $h_4n_2n_4n_{18}$, $h_6n_2n_4n_{18}$ \\ \hline
  8  & $w_1w_4w_6w_{69}$ & $w_2w_5w_{61}w_{97}$ & $n_2n_5n_{61}n_{97}$, $n_2n_5n_{61}n_{97}$ \\ \hline
  9  & $w_1w_2w_3w_5$ & $w_2w_5w_3w_{99}$ & $n_2n_5n_3n_{99}$, \\
     & & & $h_4n_2n_5n_3n_{99}$ \\ \hline
  11 & $w_1w_4w_6w_3$ & $w_2w_{97}w_4w_{18}$ & $h_4n_2n_{97}n_4n_{18}$,  \\
     & & & $h_6n_2n_{97}n_4n_{18}$ \\ \hline
  13 & $w_3w_2w_5w_4$  & $w_3w_{99}w_{88}w_{95}$ & $n_3n_{99}n_{88}n_{95}$, \\ 
     & & & $n_3n_{99}n_{88}n_{95}$ \\  \hline
  14 & $w_3w_2w_4w_{18}$ & $w_3w_2w_4w_{18}$ & $h_4n_3n_2n_4n_{18}$,  \\  
     & & &     $n_3n_2n_4n_{18}$        \\ \hline
  16 & $w_1w_4w_6w_3w_{69}$ & $w_2w_3w_{61}w_{98}w_{99}$ & $n_2n_3n_{61}n_{98}n_{99}$,  \\ 
     & & &    $h_4n_2n_3n_{61}n_{98}n_{99}$          \\ \hline
  19 & $w_2w_5w_3w_4w_6$ & $w_3w_7w_4w_{18}w_{98}$ & $h_4n_3n_7n_4n_{18}n_{98}$,  \\ 
     & & &    $h_4n_3n_7n_4n_{18}n_{98}$        \\ \hline
  20 & $w_{26}w_5w_4w_3w_2$  & $w_2w_3w_4w_{18}w_{102}$ & $h_4n_2n_3n_4n_{18}n_{102}$,  \\ 
     & & &  $h_6n_2n_3n_4n_{18}n_{102}$           \\ \hline
  26 & $w_1w_4w_6w_3w_7$ & $w_2w_7w_4w_{18}w_{102}$ & $h_4n_2n_7n_4n_{18}n_{102}$, \\ 
     & & &        $h_6n_2n_7n_4n_{18}n_{102}$      \\ \hline
  28 & $w_3w_2w_4w_{18}w_7$ & $w_3w_2w_4w_{18}w_7$ & $h_4n_3n_2n_4n_{18}n_7$,  \\ 
     & & &    $n_3n_2n_4n_{18}n_7$         \\ \hline
  30 & $w_{46}w_3w_5w_1w_4w_6$ & $w_2w_3w_{120}w_{86}w_{87}w_{99}$ & $h_4n_2n_3n_{120}n_{86}n_{87}n_{99}$,  \\ 
     & & &    $n_2n_3n_{120}n_{86}n_{87}n_{99}$      \\ \hline
  31 & $w_2w_3w_5w_7$ & $w_2w_5w_{61}w_{98}$ & $n_2n_5n_{61}n_{98}$,  \\ 
     & & &     $h_4h_8n_2n_5n_{61}n_{98}$        \\ \hline
  32 & $w_{74}w_3w_2w_5w_4$ & $w_3w_{61}w_{4}w_{18}w_{98}$ & $h_4n_3n_{61}n_{4}n_{18}n_{98}$, \\ 
     & & &    $h_1h_2n_3n_{61}n_4n_{18}n_{98}$          \\ \hline
  33 & $w_8w_1w_4w_6w_3$ & $w_3w_7w_5w_{61}w_{99}$ & $n_3n_7n_5n_{61}n_{99}$,  \\ 
     & & &   $h_4n_3n_7n_5n_{61}n_{99}$          \\ \hline
  35 & $w_1w_2w_3w_6w_8w_7$ & $w_2w_7w_{61}w_4w_{18}w_{98}$ & $h_4n_2n_7n_{61}n_4n_{18}n_{98}$,  \\ 
     & & &     $h_6n_2n_7n_{61}n_4n_{18}n_{98}$        \\ \hline
  37 & $w_4w_8w_2w_5w_7w_{120}$ & $w_2w_3w_7w_4w_{18}w_{61}$ & $h_4n_2n_3n_7n_4n_{18}n_{61}$, \\ 
     & & &   $n_2n_3n_7n_4n_{18}n_{61}$          \\ \hline
  42 & $w_2w_3w_4w_5w_6w_8$ & $w_3w_7w_{61}w_4w_{18}w_{98}$ & $h_4n_3n_7n_{61}n_4n_{18}n_{98}$,  \\ 
     & & &        $n_3n_7n_{61}n_4n_{18}n_{98}$     \\ \hline
  48 & $w_2w_4w_5w_6w_7w_8w_{120}$ & $w_2w_3w_7w_4w_{18}w_{61}w_{98}$ & $h_4n_2n_3n_7n_4n_{18}n_{61}n_{98}$,  \\ 
     & & &     $n_2n_3n_7n_4n_{18}n_{61}n_{98}$         \\ \hline
\end{longtable}

\noindent The remaining cases are covered by the following lemma.

\begin{lemma} Let $w$ or $w_0$ be one of the following:
$w_1w_4w_3w_7w_6w_8$, $w_2w_3w_4w_8w_{120}w_{18}$, $w_2w_3w_4w_7w_{120}w_8w_{18}$,
$w_2w_3w_4w_7w_{120}w_{18}w_8w_{74}$. Assume that $T$ is a maximal torus
which corresponds to the conjugacy class of $w$. Then $T$ does not have a complement
in $N(G,T)$ and $w$ has a lift to $N(G,T)$ of order $|w|$.
\end{lemma}

\begin{proof}
We consider each case for $w$ separately.

\textbf{Torus 36.} In this case $w=w_1w_4w_3w_7w_6w_8$ and $$C_W(w)=\langle w_1w_4w_3,w_7w_6w_8,w_0, w_6w_8w_{69}w_{91},w_{20}w_{24}w_{29}w_{35} \rangle\simeq(\mathbb{Z}_4\times \mathbb{Z}_4\times \mathbb{Z}_2\times \mathbb{Z}_2)\rtimes \mathbb{Z}_2,$$
where $$\langle w_1w_4w_3,w_7w_6w_8,w_0, w_6w_8w_{69}w_{91}\rangle=\langle w_1w_4w_3\rangle\times\langle w_7w_6w_8\rangle\times \langle w_0\rangle\times \langle w_6w_8w_{69}w_{91}\rangle\simeq(\mathbb{Z}_4\times \mathbb{Z}_4\times \mathbb{Z}_2\times \mathbb{Z}_2).$$

Put $n=n_1n_4n_3n_7n_6n_8$. Let $N_1=H_1n_1n_4n_3$, $N_2=H_2n_7n_6n_8$ and $N_3=H_3n_6n_8n_{69}n_{91}$ be preimages of $w_1w_4w_3$, $w_7w_6w_8$ and $w_6w_8w_{69}w_{91}$ in $K$, 
where $H_1=(\alpha_i)$, $H_2=(\beta_i)$, and $H_3=(\mu_i)$.

Using MAGMA, we see that $[n,n_1n_4n_3]=[n,n_7n_6n_8]=[n,n_6n_8n_{69}n_{91}]=1$. By Lemma~\ref{normalizer}, we have $n_1n_4n_3, n_7n_6n_8, n_6n_8n_{69}n_{91}\in \overline{N}_{\sigma n}$ and  $H_1,H_2,H_3\in\overline{T}_{\sigma n}$.

Since $(w_7w_6w_8)^4=1$, we have $N_2^4=1$. Using MAGMA, we see that $(n_7n_6n_8)^4=h_6h_8$ and hence Lemma~\ref{conjugation} implies that
$$ (Hn_7n_6n_8)^4=(\lambda_1^4, \lambda_2^4, \lambda_3^4, \lambda_4^4, \lambda_5^4, -\lambda_5^3,\lambda_5^2,-\lambda_5).$$

Therefore, $\beta_5=-1$ and $\beta_1^4=1$.

Since $[n_1n_4n_3,n_7n_6n_8]=1$, Lemma~\ref{commutator} yields
$H_1^{-1}H_1^{n_7n_6n_8}=H_2^{-1}H_2^{n_1n_4n_3}$. Using MAGMA, we see that
$$H^{-1}H^{n_7n_6n_8}=(1,1,1,1,1,\lambda_5\lambda_6^{-2}\lambda_7,\lambda_5\lambda_6^{-1}\lambda_8^{-1}, \lambda_7\lambda_8^{-2}),$$
$$H^{-1}H^{n_1n_4n_3}=(\lambda_1^{-1}\lambda_3^{-1}\lambda_4, 1, \lambda_1\lambda_3^{-2}\lambda_4, \lambda_1\lambda_2\lambda_3^{-1}\lambda_4^{-1}\lambda_5, 1, 1, 1, 1).$$
Therefore, we have $\beta_4=\beta_1\beta_3$, $\beta_3^2=\beta_1\beta_4$, and $\beta_1\beta_2\beta_5=\beta_3\beta_4$.
Hence, $\beta_3^2=\beta_1\beta_4=\beta_1(\beta_1\beta_3)$. So $\beta_3=\beta_1^2$ and $\beta_4=\beta_1\beta_3=\beta_1^3$. Since $\beta_5=-1$ and $\beta_1^4=1$,  we have $\beta_1\beta_2(-1)=\beta_3\beta_4=\beta_1^5=\beta_1$. Therefore, $\beta_2=-1$.

Since $[n_6n_8n_{69}n_{91},n_7n_6n_8]=1$, Lemma~\ref{commutator} yields $H_3^{-1}H_3^{n_7n_6n_8}=H_2^{-1}H_2^{n_6n_8n_{69}n_{91}}$. Using MAGMA, we see that
$$H^{-1}H^{n_6n_8n_{69}n_{91}}=(\lambda_2^{-2}\lambda_5, *,*,*,*,*,*,*).$$

Therefore, $\beta_2^{-2}\beta_5=1$; a contradiction with 
$\beta_2=\beta_5=-1$.

Calculations in MAGMA show that $(h_5n)^{4}=1$. Thus, $h_5n$ is a required lift of $w$ to $\overline{N}_{\sigma h_5n}$.

\textbf{Torus 41.}  In this case $w=w_2w_3w_4w_8w_{120}w_{18}$ and $C_W(w)=\langle w^2w_0 \rangle\times\langle i,j,k \rangle\simeq \mathbb{Z}_{6}\times (\SL_2(3):\mathbb{Z}_4)$,
where $i=w_6w_{19}w_{26}$, $j=w_4w_{13}w_{40}$, and $k=w_1w_2w_4w_6w_{48}w_{51}$. Using GAP, one can verify that $\langle i,j,k \rangle\simeq\langle a,b,c~|~a^4=b^4=c^3=[a,b]=ca^3b^2cb^3=1\rangle$
and $i$, $j$, $k$ satisfy these relations.

Put $n=n_2n_3n_4n_8n_{120}n_{18}$ and $a=h_6n_6n_{19}n_{26}$.
Then, using MAGMA, we see that $[n,a]=1$ and hence $a\in \overline{N}_{\sigma n}$. Suppose that there exists
 a complement $K$ for $\overline{T}_{\sigma n}$. Let $N_1=H_1n$, $N_2=H_2a$, and $N_0=H_0h_2h_3h_5n_0$ be preimages of $w$, $i$ and, $w_0$ in $K$ with $H_1=(\mu_i)$, $H_0=(\beta_i)$ and $H_2=(\alpha_i)$.

Since $[n,a]=1$, Lemma~\ref{conjugation} implies that $H_1^{-1}H_1^a=H_2^{-1}H_2^n$. By Lemma~\ref{conjugation}, we find that
$$H^{-1}H^n=(\lambda_8^{-2},\lambda_2^{-1}\lambda_3\lambda_4^{-1}\lambda_5\lambda_8^{-3},\lambda_1\lambda_4^{-1}\lambda_6\lambda_8^{-4},
\lambda_3^2\lambda_4^{-2}\lambda_6\lambda_8^{-6},\lambda_2^{-1}\lambda_3\lambda_5^{-1}\lambda_6\lambda_8^{-5},\lambda_8^{-4},\lambda_8^{-3},\lambda_7\lambda_8^{-3}),$$
$$H^{-1}H^a=(1,\lambda_2^{-1}\lambda_3\lambda_6^{-1}\lambda_7,\lambda_1\lambda_5^{-1}\lambda_7,
\lambda_1\lambda_2^{-1}\lambda_3\lambda_5^{-1}\lambda_6^{-1}\lambda_7^2,
\lambda_1\lambda_2^{-1}\lambda_3\lambda_5^{-1}\lambda_6^{-1}\lambda_7^2,
\lambda_1\lambda_6^{-2}\lambda_7^2,1,1).$$
Therefore, we conclude that $\alpha_7\alpha_8^3=\alpha_8^3=\alpha_8^2=1$. So $\alpha_7=\alpha_8=1$.
Then  
\begin{multline*}
(1,\mu_2^{-1}\mu_3\mu_6^{-1}\mu_7,\mu_1\mu_5^{-1}\mu_7,
\mu_1\mu_2^{-1}\mu_3\mu_5^{-1}\mu_6^{-1}\mu_7^2,
\mu_1\mu_2^{-1}\mu_3\mu_5^{-1}\mu_6^{-1}\mu_7^2,\mu_1\mu_6^{-2}\mu_7^2,1,1)=\\=
(1,\alpha_2^{-1}\alpha_3\alpha_4^{-1}\alpha_5,\alpha_1\alpha_4^{-1}\alpha_6,\alpha_3^2\alpha_4^{-2}\alpha_6,\alpha_2^{-1}\alpha_3\alpha_5^{-1}\alpha_6,1,1,1).
\end{multline*}
Since the product of the second and the third coordinates on the left side equals the fourth coordinate, 
we have $(\alpha_2^{-1}\alpha_3\alpha_4^{-1}\alpha_5)(\alpha_1\alpha_4^{-1}\alpha_6)=\alpha_3^2\alpha_4^{-2}\alpha_6$ and hence $\alpha_1\alpha_5=\alpha_2\alpha_3$. Moreover, the fourth and the fifth coordinates coincide on the left side, so $\alpha_3^2\alpha_4^{-2}\alpha_6=\alpha_2^{-1}\alpha_3\alpha_5^{-1}\alpha_6$ and hence
$\alpha_2\alpha_3=\alpha_4^2\alpha_5^{-1}.$

Using MAGMA, we see that $a^4=h_2h_3$ and hence Lemma~\ref{conjugation} implies that $$(Ha)^4=(\lambda_1^4, -\lambda_1\lambda_2^2\lambda_3^2\lambda_5^{-2}\lambda_7^2,-\lambda_1^3\lambda_2^2\lambda_3^2\lambda_5^{-2}\lambda_7^2, \lambda_1^4\lambda_4^4\lambda_5^{-4}\lambda_7^4,\lambda_1^4\lambda_7^4,\lambda_1^2\lambda_7^4,\lambda_7^4,\lambda_8^4).$$
Therefore, we infer that $-\alpha_1^{-1}=\alpha_2^2\alpha_3^2\alpha_5^{-2}\alpha_7^2$, $\alpha_1^2=1$.
Since $\alpha_7=1$, we have $-\alpha_1\alpha_5^2=\alpha_2^2\alpha_3^2$. However, we know that 
$\alpha_1\alpha_5=\alpha_2\alpha_3$ and hence $\alpha_1=-1$. Now from 
$\alpha_2\alpha_3=\alpha_4^2\alpha_5^{-1}$, we obtain $\alpha_5^2=-\alpha_4^2$.

Calculations in MAGMA show that $[a,n_0]=1$. By Lemma~\ref{commutator}, we infer that
$H_0^{-1}H_0^a=H_2^{-1}H_2^{n_0}$. Then $H_2^{-1}H_2^{n_0}=(\alpha_1^{-2},\alpha_2^{-2},\alpha_3^{-2},\alpha_4^{-2},\alpha_5^{-2},\alpha_6^{-2},\alpha_7^{-2},\alpha_8^{-2})$. Applying the above equation for $H^{-1}H^a$, we see that
$H_0^{-1}H_0^a=H_0^{-2}=(\ast,\ast,\ast,\beta_1\beta_2^{-1}\beta_3\beta_5^{-1}\beta_6^{-1}\beta_7^2,
\beta_1\beta_2^{-1}\beta_3\beta_5^{-1}\beta_6^{-1}\beta_7^2,\ast,\ast,\ast).$
Since $\alpha_4^2=-\alpha_5^2$, we arrive at a contradiction.

Using MAGMA, we see that $n^{12}=1$ and so $n$ is a required lift of $w$ to $\overline{N}_{\sigma n}$. Then $nn_0$ is a required lift of $ww_0$.

\textbf{Torus 49.} In this case $w=w_2w_3w_4w_7w_{120}w_8w_{18}$ and $$C_W(w)=\langle w_0 \rangle \times
\langle w \rangle \times \langle w_{11}w_{12}, w_4w_{17}, w_{76}w_{86} \rangle \simeq \mathbb{Z}_2\times \mathbb{Z}_4\times (((\mathbb{Z}_4 \times \mathbb{Z}_4) : \mathbb{Z}_3) : \mathbb{Z}_2).$$

Put $n=h_4n_2n_3n_4n_7n_{120}n_8n_{18}$. Observe that $w_{15}w_{119}w_8\in C_W(w)$. Using MAGMA, we see that $[h_2n_{15}n_{119}n_{8},n]=1$ and hence $h_2n_{15}n_{119}n_{8}\in \overline{N}_{\sigma n}$.

Suppose that exists a complement $K$ for $\overline{T}_{\sigma n}$ in $\overline{N}_{\sigma n}$. Then $N_0=H_0n_0$, $N_1=H_1n$, $N_2=H_2h_2n_{15}n_{119}n_{8}$ belong to $K$ for some $H_0$, $H_1=(\mu_i)$ and $H_2=(\alpha_i)$.

Since $[n, h_2n_{15}n_{119}n_{8}]=1$, Lemma~\ref{commutator} implies that
$H_2^{-1}H_2^{N_1}=H_1^{-1}H_1^{N_2}$. By Lemma~\ref{conjugation},
\begin{multline*}
H^{-1}H^{N_1}=(\lambda_7^{-2}\lambda_8^2,\lambda_2^{-1}\lambda_3\lambda_4^{-1}\lambda_5\lambda_7^{-3}\lambda_8^3, \lambda_1\lambda_4^{-1}\lambda_6\lambda_7^{-4}\lambda_8^4,\lambda_3^2\lambda_4^{-2}\lambda_6\lambda_7^{-6} \lambda_8^6,\lambda_2^{-1}\lambda_3\lambda_5^{-1}\lambda_6\lambda_7^{-5}\lambda_8^5, \lambda_7^{-4}\lambda_8^4,\\
\lambda_6\lambda_7^{-4}\lambda_8^2,\lambda_7^{-1}),    
\end{multline*}
$$H^{-1}H^{N_2}=(\lambda_8^{-2},\lambda_8^{-3},\lambda_8^{-4},\lambda_8^{-6},\lambda_8^{-5},\lambda_8^{-4},
\lambda_6\lambda_7^{-2}\lambda_8^{-2},\lambda_6\lambda_7^{-1}\lambda_8^{-2}).$$

Applying this to $H_2^{-1}H_2^{N_1}=H_1^{-1}H_1^{N_2}$, we see that the third and the sixth coordinates on the right side coincide, so $\alpha_1\alpha_4^{-1}\alpha_6\alpha_7^{-4}\alpha_8^4=\alpha_7^{-4}\alpha_8^4$ and hence
$\alpha_1\alpha_4^{-1}\alpha_6=1$. Moreover, we see that the square of the second coordinate equals the fourth coordinate, so $(\alpha_2^{-1}\alpha_3\alpha_4^{-1}\alpha_5\alpha_7^{-3}\alpha_8^3)^2=\alpha_3^2\alpha_4^{-2}\alpha_6\alpha_7^{-6}\alpha_8^6$
and hence $\alpha_2^{-2}\alpha_5^2=\alpha_6$. Finally, since the fifth coordinate equals the product of the first and the second coordinates,
we infer that $(\alpha_7^{-2}\alpha_8^2)(\alpha_2^{-1}\alpha_3\alpha_4^{-1}\alpha_5\alpha_7^{-3}\alpha_8^3)=\alpha_2^{-1}\alpha_3\alpha_5^{-1}\alpha_6\alpha_7^{-5}\alpha_8^5$. Whence, $\alpha_4^{-1}\alpha_5^2=\alpha_6$.

Since $[h_2n_{15}n_{119}n_8,n_0]=1$, we conclude that $$H_0^{-1}H_0^{N_2}=H_2^{-1}H_2^{N_0}=H_2^{-2}=(\alpha_1^{-2},\alpha_2^{-2},\alpha_3^{-2},\alpha_4^{-2},\alpha_5^{-2},\alpha_6^{-2},\alpha_7^{-2},\alpha_8^{-2}).$$ 
We use the same equalities of coordinates for the right side as in the previous
paragraph and get that $\alpha_3^2=\alpha_6^2$, $\alpha_2^4=\alpha_4^2$, and $\alpha_1^2\alpha_2^2=\alpha_5^2$.
We obtain above $\alpha_5^2\alpha_2^{-2}=\alpha_6$, so $\alpha_1^2=\alpha_6$.

Calculations in MAGMA show that $(h_2n_{15}n_{119}n_{8})^4=h_2h_5$. By Lemma~\ref{conjugation},
$$(Hh_2n_{15}n_{119}n_{8})^4=(\lambda_1^4\lambda_6^{-2},-\lambda_2^4\lambda_6^{-3}, \lambda_3^4\lambda_6^{-4},
\lambda_4^4\lambda_6^{-6},-\lambda_5^4\lambda_6^{-5},1,1,1).$$
Therefore, $\alpha_2^4=-\alpha_6^3$. On the other hand, we see above that  $\alpha_2^4=\alpha_4^2$ and hence
$\alpha_4^2=-\alpha_6^3$. Squaring up the equation $\alpha_1\alpha_4^{-1}\alpha_6=1$, we obtain $\alpha_1^2(-\alpha_6^{-1})=1$; a contradiction with the equation $\alpha_1^2=\alpha_6$.

Since $(h_6n)^4=1$ and hence $h_6n$ is a required lift of $w$ in $\overline{N}_{\sigma h_6n}$. Then $h_6nn_0$ is a required lift of $ww_0$.

\textbf{Torus 59.} In this case $w=w_2w_3w_4w_7w_{120}w_{18}w_8w_{74}$ and 
$$C_W(w)=\langle a,b,c,d,e\rangle\simeq (\mathbb{Z}_4.((\mathbb{Z}_2\times \mathbb{Z}_2 \times \mathbb{Z}_2 \times \mathbb{Z}_2) : A_6)) : \mathbb{Z}_2.$$
Moreover, $C_W(w)$ is isomorphic to the following group:
\begin{multline*}\langle a,b,c,d,e~|~a^2=b^2=c^2=d^2=e^2=
(ab)^2=(ac)^3=(ad)^3=(ae)^4=\\=(bc)^4=(bd)^3=(be)^2=(cd)^2=(ce)^3=(de)^2=aceaecbacb=1\rangle,
\end{multline*}
and the elements $w_1w_{99}$, $w_2w_5$, $w_4w_{17}$, $w_6w_{35}$, and $w_9w_{79}$ of $C_W(w)$ satisfy this set of relation.

Put $n=h_4n_2n_3n_4n_7n_{120}n_{18}n_8n_{74}$. Using MAGMA, we see that $[n,h_3h_5n_1n_{99}]=[n,h_4h_7n_2n_5]=[n,h_3h_8n_9n_{79}]=1$. Therefore, $h_3h_5n_1n_{99}, h_4h_7n_2n_5, h_3h_8n_9n_{79}\in \overline{N}_{\sigma n}$. Let $N_1,N_2,N_3,N_4$ be preimages of $w$, $a$, $b$, and $e$ in $K$, respectively.  Then $N_1=H_1n$, $N_2=H_2h_3h_5n_1n_{99}$, $N_3=H_3h_4h_7n_2n_5$ and $N_4=H_4h_3h_8n_9n_{79}$, where
$H_1=(\mu_i)$, $H_2=(\alpha_i)$, $H_3=(\beta_i)$, and $H_4=(\delta_i)$, are elements of $\overline{T}_{\sigma n}$.

Since $b^2=1$, we have $N_3^2=1$. Using MAGMA, we see that $(h_4h_7n_2n_5)^2=1$. By Lemma~\ref{conjugation},
$$(HN_3)^2=(\lambda_1^2,\lambda_4,\lambda_3^2,\lambda_4^2,\lambda_4\lambda_6,\lambda_6^2,\lambda_7^2,\lambda_8^2).$$
Therefore, $\beta_4=\beta_6=1$ and $\beta_1^2=\beta_3^2=\beta_7^2=\beta_8^2=1$.

Since $[a,b]=1$ and $[h_3h_5n_1n_{99},h_4h_7n_2n_5]=1$, Lemma~\ref{commutator} implies that $H_3^{-1}H_3^{N_2}=H_2^{-1}H_2^{N_3}$.
By Lemma~\ref{conjugation}, 
$$H^{-1}H^{N_2}=(\lambda_1^{-2}\lambda_3^2\lambda_4^{-1}\lambda_6\lambda_7^{-1},t^2,
t^2, t^4, t^3, t^2, t^2, t),\text{ where } t=\lambda_3\lambda_4^{-1}\lambda_6\lambda_7^{-1};$$
$$H^{-1}H^{N_3}=(1,\lambda_2^{-2}\lambda_4,1,1,\lambda_4\lambda_5^{-2}\lambda_6,1,1,1).$$
Applying to $H_2$ and $H_3$, we see that $\beta_1^{-2}\beta_3^2\beta_4^{-1}\beta_6\beta_7^{-1}= \beta_3\beta_4^{-1}\beta_6\beta_7^{-1}=1$. Since $\beta_4=\beta_6=\beta_1^2=\beta_3^2=1$,
we infer that $\beta_3=1$ and hence $\beta_7=1$ as well.

Since $[w,b]=1$, Lemma~\ref{conjugation} implies that $H_3^{-1}H_3^{N_1}=H_1^{-1}H_1^{N_3}$.
By Lemma~\ref{conjugation}, 
$$H^{-1}H^{N_1}=(\ast,\ast,\ast,\ast,\ast,\ast,\ast,\lambda_1^{-1}).$$

Applying this to $H_3$ and using the above equation for $H^{-1}H^{N_3}$, we find that $(\ast,\ast,\ast,\ast,\ast,\ast,\ast,\beta_1^{-1})=(\ast,\ast,\ast,\ast,\ast,\ast,\ast,1)$ and hence $\beta_1=1$.

Since $[b,e]=1$ and $[h_4h_7n_2n_5, h_3h_8n_9n_{79}]=h_2h_8$, Lemma~\ref{conjugation} implies that
$H_3^{-1}H_3^{N_4}=H_4^{-1}H_4^{N_3}h_2h_8$.
By Lemma~\ref{conjugation},
$$H^{-1}H^{N_4}=(\lambda_1^{-2}\lambda_4\lambda_7^{-1},\lambda_1^{-1}\lambda_3\lambda_7^{-1},
\lambda_1^{-2}\lambda_4\lambda_7^{-1},\lambda_1^{-2}\lambda_3^2\lambda_7^{-2},\lambda_1^{-2}\lambda_3^2\lambda_7^{-2},\lambda_1^{-2}\lambda_3^2\lambda_7^{-2},\lambda_1^{-2}\lambda_3^2\lambda_7^{-2},\lambda_1^{-1}\lambda_3\lambda_7^{-1}).$$
Since $\beta_1=\beta_3=\beta_4=\beta_7=1$, we get $H_3^{-1}H_3^{N_4}=1$. On the other hand, we apply the above equation for 
$H_4^{-1}H_4^{N_3}$ and find that $h_2h_8=(1,\delta_2^{-2}\delta_4,1,1,\delta_4\delta_5^{-2}\delta_6,1,1,1)$;
a contradiction.

Using MAGMA, we see that $n^4=1$, so $n$ is a lift of $w$ in $\overline{N}_{\sigma n}$ of the same order. Then $nn_0$ is a required lift of $ww_0$.

\end{proof}

\subsection{Type $E_8$: splitting cases}
We consider all cases in a similar manner. First, we need the following lemma for maximal tori of odd order.
\begin{lemma}\label{l:split_E8_odd} Let $w\in W$ and elements $x_1$,$x_2$,$\ldots$,$x_m$ generate $C_W(w)$.
Suppose $n\in\overline{N}$ such that $\pi(n)=w$ and there exist $y_1,y_2,\ldots,y_m\in\mathcal{T}$ 
such that $\pi(y_i)=x_i$. If $|\overline{T}_{\sigma n}|$ is odd then $\langle y_1,y_2,\ldots,y_m \rangle$ is a complement for $\overline{T}_{\sigma n}$ in $\overline{N}_{\sigma n}$.
\end{lemma}
\begin{proof}
Since $x_1, x_2,\ldots,x_m$ generate $C_W(w)$, 
there exists a set of defining relations $S$ of $C_W(w)$ such that $x_1, x_2,\ldots,x_m$ satisfy $S$.
Therefore, each relation of $S$ holds for $y_1,y_2,\ldots,y_m$ up to some elements of $\overline{T}_{\sigma n}$. Let $h$ be any of such elements. 
Since $y_1,y_2,\ldots,y_m\in\mathcal{T}$, we infer that $h\in\mathcal{T}$.
Since $\mathcal{T}\cap\overline{T}_{\sigma n}\leqslant\mathcal{H}$, we have $h\in\overline{T}_{\sigma n}\cap\mathcal{H}$. However, $\mathcal{H}$ is an elementary abelian 2-group and hence
$\mathcal{H}\cap{\overline{T}_{\sigma n}}=1$. Thus, all relations of $S$ hold in $\langle y_1,y_2,\ldots,y_m\rangle$,
and hence it is a complement for $\overline{T}_{\sigma n}$.
\end{proof}

Our strategy is similar in all cases. 
For an element $w$, we find $n\in\mathcal{T}$ such that $\pi(n)=w$ and a set of relations that defines $C_W(w)$. Then we present
elements in $\overline{N}_{\sigma n}$ that satisfy this set of relations. All data is listed in Tables~\ref{t:split_E8}.
As an example, we consider $w=w_1w_5w_3w_6$ that corresponds to Torus 10 in Table~\ref{t:main:E8}.
In this case $$C_W(w)\simeq\mathbb{Z}_3\times((\mathbb{Z}_2\times S_3\times S_3\times S_3):\mathbb{Z}_2).$$
One can verify in GAP that $C_W(w)$ is isomorphic the group defined as follows:
\begin{multline*}
\langle a,b,c,d,e,f,g,i,j~|~ a^3=b^2=c^2=d^2=e^2=f^2=g^2=i^2=j^2=R(a)=R(b)=[c,e]=\\=[c,f]=[c,g]=[c,i]=[d,e]=[d,f]=[d,g]=[d,i]=[e,g]=[e,i]=[f,g]=[f,i]=\\=(cd)^3=(ef)^3=(gi)^3=jcje=jdjf=jgjbg= jijbi=1\rangle 
\end{multline*}
We replace the set of relations defining that $x$ commutes with all other generators by $R(x)$.
Now, if we take $a=w$, $b=w_0$, $c=w_2$, $d=w_{69}$, $e=w_8$, $f=w_{120}$,
$g=w_2w_{37}w_{40}w_{57}$, $i=w_2w_{32}w_{51}w_{52}$, and $j=w_1w_{34}w_{36}w_{84}$ then
all these elements lie in $C_W(w)$ and satisfy the above relations for this group.
Finally, we put $a=n_1n_5n_3n_6$, $b=n_0$, $c=h_{69}n_{2}$, $d=h_2n_{69}$, $e=h_{120}n_8$,
$f=h_8n_{120}$, $g=h_1h_6n_2n_{37}n_{40}n_{57}$, $i=h_1h_3h_6n_2n_{32}n_{51}n_{52}$,
$j=h_2h_8n_1n_{34}n_{36}n_{84}$, and $K=\langle a,b,c,d,e,f,g,i,j \rangle$. Then $\pi(K)=C_W(w)$.
Computations in MAGMA show that generators of $K$ satisfy the set of relations above, so
$K\simeq C_W(w)$. On the other hand, $a$ commutes with other generators of $K$, so 
$a,b,c,d,e,f,g$ lie in $\overline{N}_{\sigma{a}}$ by Lemma~\ref{normalizer}.
Therefore, $\overline{N}_{\sigma{a}}$ splits over $\overline{T}_{\sigma{a}}$.

Other cases can be verified in the same way. We divide information into two parts.
Table~\ref{t:split_E8} contains information for maximal tori of even order.
The first column of Table~\ref{t:split_E8} contains numbers of tori in accordance with Table~\ref{t:main:E8}. The second column for each $w$ contains a set of relations $S(w)$ that defines $C_W(w)$. As above, we replace the set of relations defining that $x$ commutes with all other generators by $R(x)$.
The third column contains examples of generators of a complement. All such generators lie in $\mathcal{T}$.
Therefore, it is easy to verify in MAGMA that the generators satisfy relations $S(w)$. The natural preimage of $w$ in $\mathcal{T}$ is denoted by $n$. In each case we choose an element $x$ that defines the group $\overline{T}_{\sigma x}$. Usually $x=n$ but sometimes they differ. To verify that a generator $y$ lies in $\overline{N}_{\sigma x}$, one can check that $[x,y]=1$ and apply Lemma~\ref{normalizer}.

In Table~\ref{t:split_E8_odd}, we list information for maximal tori of odd order.
The first column of this table contains numbers of maximal tori
in accordance with Table~\ref{t:main:E8}. The second column contains examples of generators of a complement. All such generators lie in $\mathcal{T}$. The natural preimage of $w$ in $\mathcal{T}$ is denoted by $n$. Using MAGMA,
we see that for every generator $y$ it is true that $[n,y]=1$ and hence in each case $y\in\overline{N}_{\sigma n}$ by Lemma~\ref{normalizer}. Now 
Lemma~\ref{l:split_E8_odd} implies that in all these cases listed elements generate corresponding complements.

For convenience, we add all verified equations in~\cite{StarE8}.

\begin{longtable}{|c|l|l|}
\caption{Splitting normalizers of maximal tori of even order in~$E_8(q)$}\label{t:split_E8} \\
\hline
 \No  & Defining relations for $C_W(w)$  & Generators for a complement \\ 
  &  & of $\overline{T}_{\sigma{x}}$ in $\overline{N}_{\sigma{x}}$   \\  
\hline\endfirsthead
\multicolumn{3}{c}{\textit{Table~\ref{t:split_E8} (continued)}} \\ \hline
 \No  & Defining relations for $C_W(w)$  & Generators for a complement \\ 
  &  & of $\overline{T}_{\sigma{x}}$ in $\overline{N}_{\sigma{x}}$   \\ 
\hline\endhead
\hline
\endfoot
\hline\endlastfoot
  10  & $a^3=b^2=c^2=d^2=e^2=f^2=1$,  & $a=x=n$, $b=n_0$, \\
      & $g^2=i^2=j^2=1$, $R(a)$, $R(b)$, & $c=h_{69}n_2$, $d=h_2n_{69}$,  \\
      &  $[c,e]=[c,f]=[c,g]=[c,i]=1$,  &  $e=h_{120}n_8$, $f=h_8n_{120}$, \\
      & $[d,e]=[d,f]=[d,g]=[d,i]=1$, & $g=h_1h_6n_2n_{37}n_{40}n_{57}$,  \\
      & $[e,g]=[e,i]=[f,g]=[f,i]=1$,  & $i=h_1h_3h_6n_2n_{32}n_{51}n_{52}$, \\
      & $(cd)^3=(ef)^3=(gi)^3=1$,  &   $j=h_2h_8n_1n_{34}n_{36}n_{84}$ \\
      & $jcje=jdjf=jgjbg=jijbi=1$  & \\ \hline
  12  & $a^{10}=b^2=c^2=d^2=e^2=1$,  & $x=n$, $a=nn_0$, $b=h_2h_5n_6$, \\
      & $R(a)$, $(bc)^3=(bd)^2=(be)^2=1$, & $c=h_2h_5h_6n_7$, $d=h_2h_5h_6h_7n_8$,  \\ 
      & $(cd)^3=(ce)^2=(de)^3=1$ & $e=h_6h_8n_{120}$ \\ \hline
  15 & $a^6=b^2=c^2=d^2=e^2=f^2=1$, & $a=x=h_4n$, $b=n_0$,   \\
     &  $R(a)$, $R(b)$, $(cd)^3=(ef)^3=1$,  & $c=h_{120}n_8$, $d=h_8n_{120}$,   \\ 
     & $(ce)^2=(cf)^2=(de)^2=(df)^2=1$ & $e=h_1h_4h_6n_{32}n_{51}n_{52}$,  \\ 
     & & $f=h_4n_{37}n_{40}n_{57}$ \\ \hline
  17 & $a^{10}=b^2=c^2=d^2=(cd)^3=1$,  &  $a=x=h_2n$, $b=n_0$,  \\ 
   & $R(a)$, $R(b)$ &  $c=h_2h_5h_7n_8$, $d=h_8n_{120}$ \\ \hline
  18 & $a^6=b^2=c^2=d^2=e^2=1$,  & $a=x=h_2n$, $b=n_0$, \\
     & $(de)^3=[c,d]=[c,e]=1$, &  $c=h_2h_3h_5n_{69}$, $d=h_2h_5h_7n_8$, \\ 
     & $R(a)$, $R(b)$,  & $e=h_8n_{120}$ \\ \hline
  21 & $a^2=b^2=c^2=d^{12}=e^6=1$,   & $x=n$, $a=n_0$, $b=h_{120}n_8$,\\
     & $(bc)^3=[b,d]=[b,e]=1$, & $c=h_8n_{120}$,  \\
     & $[c,d]=[c,e]=1$, $R(a)$,  &  $d=h_4h_5n_1n_2n_6n_4n_{17}n_{26}$, \\ 
     & $[d^8,e]=(d^6e^{-1})^3=1$,  & $e=h_6n_1n_2n_6n_{18}n_{33}n_{45}$ \\ 
     & $d^6e^2d^6e^{-2}=ed^8(d^{-1}e)^2d^{-1}=1$ & \\ \hline
  22 & $a^6=b^2=c^2=d^2=e^2=f^2=1$, & $a=x=h_1h_4h_6n$, $b=n_0$, \\
     & $R(a)$, $R(b)$, $(cd)^3=(ef)^3=1$, & $c=h_{120}n_8$, $d=h_8n_{120}$,    \\
     & $(ce)^2=(cf)^2=(de)^2=(df)^2=1$ & $e=h_{69}n_{32}$, $f=h_{32}n_{69}$ \\ \hline
  23 & $a^{12}=b^2=c^2=d^2=(cd)^3=1$,  & $a=x=n$, $b=n_0$,   \\ 
        & $R(a)$, $R(b)$  &  $c=h_{120}n_8$, $d=h_8n_{120}$  \\ \hline
  24 & $a^{18}=b^2=c^2=(bc)^3=1$, & $x=n$, $a=xn_0$, $b=h_{120}n_8$, \\ 
  & $[a,b]=[a,c]=1$ & $c=h_8n_{120}$   \\ \hline
  25 & $a^6=b^2=c^2=(bc)^3=1$, $R(a)$, & $x=n$, $a=x^2n_0$, \\
     & $d^4=e^3=1$,  & $b=h_{120}n_8$, $c=h_8n_{120}$,  \\
     & $ded^{-1}ede=(e^{-1}d)^3=1$, & $d=h_1h_2h_5n_3n_6n_{25}n_{37}$, \\ 
     &  $[b,d]=[b,e]=[c,d]=[c,e]=1$ & $e=h_1h_2h_4h_6n_1n_4n_6n_{17}n_{26}n_{57}$ \\ \hline
  27 & $a^{30}=b^2=aba^{-1}b^{-1}=1$ & $x=n$, $a=xn_0$, $b=h_6h_8n_{120}$ \\ \hline
  29 & $a^{14}=b^2=aba^{-1}b^{-1}=1$ & $x=n$, $a=xn_0$, $b=h_1h_4h_6h_8n_{120}$ \\ \hline
  34 &  $a^6=b^2=c^2=d^2=e^2=f^2=1$, &  $a=x=h_4h_7n$, $b=n_0$, \\
     & $(de)^3=[c,d]=[c,e]=1$,   &  $c=h_1h_2h_4h_6n_2$,\\
     & $R(a)$, $R(b)$, &  $d=h_3h_4n_{32}n_{51}n_{52}$, \\ 
     & $fcfa^3c=fdfa^3bd=fefa^3be=1$   & $e=h_1h_4h_6n_{37}n_{40}n_{57}$, \\ 
    & & $f=h_4h_5h_6h_7n_1n_{28}n_{42}n_{84}$ \\ \hline
  38 & $a^{10}=b^2=c^2=(bc)^4=1$, & $x=h_3h_7n$, $a=xn_0$,   \\
     & $[a,b]=[a,c]=1$ & $b=h_8n_{97}n_{98}$, $c=h_2h_5h_7n_8$  \\  \hline
  39 & $a^6=b^2=c^2=d^2=1$,  & $a=x=h_3h_8n$, $b=n_0$,  \\
     & $R(a)$, $R(b)$, $[c,d]=1$ & $c=h_3h_5h_7n_1$, $d=h_4h_5h_7h_8n_{120}$ \\ \hline
  40 & $a^6=b^6=c^2=d^2=(cd)^3=1$,  & $a=x=h_3n$, $b=n_7n_8n_0$, \\
     & $R(a)$, $[b,c]=[b,d]=1$ & $c=h_3h_5n_{23}n_{24}$, $d=h_2h_5n_{105}n_{106}$  \\ \hline
  43 & $a^2=b^{18}=c^4=1$, & $x=h_7n$, $a=n_0$,   \\
     & $b^{-1}(c^{-1}b)^2c^{-1}b^{-1}=1$, $R(a)$, & $b=h_1h_4h_5h_7n_1n_6n_{63}n_4n_8n_{19}n_{51}$, \\
     & $(b^{-1}c^{-1})^2c^{-1}(bcb^2)^3c^{-1}b^{-2}cbc=1$ & $c=h_4n_1n_4n_3n_{32}$ \\ \hline
  44 & $a^{30}=b^2=[a,b]=1$ & $a=x=h_4n$, $b=n_0$ \\ \hline
  45 & $a^{20}=b^2=[a,b]=1$ & $a=x=h_5n$, $b=n_0$ \\ \hline
  46 & $a^{14}=b^2=[a,b]=1$ & $a=x=h_8n$, $b=n_0$ \\ \hline
  47 & $a^{8}=b^2=[a,b]=1$  & $a=x=h_2n$, $b=n_0$ \\ \hline
  50 & $a^{24}=b^2=[a,b]=1$  & $a=x=n$, $b=n_0$   \\ \hline
  51 & $a^{12}=b^6=[a,b]=1$ & $a=x=h_1n$, \\
    &      & $b=h_2h_5h_7n_1n_2n_5n_8n_7n_{44}n_{71}n_{89}$\\ \hline
  52 & $a^{12}=b^2=c^2=1$, $R(a)$, $R(b)$ & $a=x=h_7n$, $b=n_0$,\\ 
        & & $c=h_2h_5h_7n_8$  \\ \hline
  53 & $a^{18}=b^2=[a,b]=1$ & $a=x=h_7n$, $b=n_0$ \\ \hline
  54 & $a^6=b^2=c^4=d^3=1$, & $a=x=h_{120}n$, $b=n_0$,  \\
     & $cdc^{-1}dcd=(d^{-1}c)^3=1$,  &  $c=h_1h_2h_5n_3n_6n_{25}n_{37}$, \\ 
     & $R(a)$, $R(b)$ & $d=h_1h_2h_4h_6n_1n_4n_6n_{17}n_{26}n_{57}$ \\ \hline
  55 & $a^{12}=b^2=[a,b]=1$ &  $a=x=h_1n$, $b=n_0$ \\ \hline
  60 & $a^{12}=b^2=c^2=(bc)^3=1$,  & $a=x=h_4n$,  \\
     & $[a,b]=[a,c]=1$ & $b=h_1h_3h_4h_6h_7n_{69}n_{70}$,\\ 
     & & $c=h_2h_7n_{78}n_{79}$ \\ \hline
  61 & $a^8=b^4=c^3=(bc)^4=1$,  & $a=x=n$, $b=h_2n_2n_{109}n_{10}n_{11}$, \\
     & $acabc^{-1}bcb^{-1}a^4b$, $R(a)$ & $c=h_2h_5h_8n_1n_2n_7n_{29}n_{18}n_{44}n_{56}n_{119}$  \\ 
\end{longtable}

\begin{longtable}{|c|l|}\caption{Splitting normalizers of maximal tori of odd order 
in $E_8(q)$}\label{t:split_E8_odd} \\
\hline
 \No  &  Generators for a complement of $\overline{T}_{\sigma{n}}$ in $\overline{N}_{\sigma{n}}$ \\ 
\hline\endfirsthead
\multicolumn{2}{c}{\textit{Table~\ref{t:split_E8_odd} (continued)}} \\
\hline
 \No  & Generators for a complement of $\overline{T}_{\sigma{n}}$ in $\overline{N}_{\sigma{n}}$ \\
\hline\endhead
\hline\endfoot
\endlastfoot
  56 & $a=n^2n_0$, $b=h_1h_4n_1n_4n_{18}n_{44}$,\\
     & $c=h_1h_3h_5h_6h_7h_8n_1n_2n_{64}n_{116}n_{26}n_{28}n_{32}n_{120}$ \\ \hline
  57 & $a=n$, $b=h_2h_3h_4h_5h_7n_1n_2n_5n_{44}n_{28}n_{45}n_{56}n_{114}$, \\
     & $c=h_1h_2h_3h_4h_5h_6n_1n_4n_6n_8n_{58}n_{63}n_{96}n_{113}$     \\ \hline
  58 & $a=nn_0$, $b=h_1h_5h_8n_{49}n_{67}$ \\ \hline
  62 & $a=n$, $b=h_4h_5h_6h_8n_1n_5n_{20}n_{71}n_{10}n_{38}n_{44}n_{67}$, \\   
     & $c=h_1h_2h_3h_5h_6h_8n_1n_5n_{20}n_{78}n_{18}n_{33}n_{49}n_{63}$, \\
     & $d=h_1h_2h_3h_6h_7h_8n_8n_{99}n_{59}n_{120}$ \\ \hline
  63 & $a=n^2$, $b=h_3h_5n_2n_{32}n_{10}n_{63}$, $c=h_2h_4n_4n_2$, \\  & $d=h_1h_2h_5h_6h_7h_8n_8n_{104}n_{58}n_{120}$, $e=n_{61}n_{67}$ \\ \hline
   64 & $a=n$ \\ \hline
   65 & $a=n$ \\ \hline
   66 & $a=n$ \\ \hline
   67 & $a=n$, $b=h_3h_5h_6n_{18}n_{45}n_{92}n_{112}$, $c=h_2h_3h_4n_2n_{29}n_{4}n_{17}$ \\ \hline
\end{longtable}

\begin{longtable}{|c|l|c|l|l|}
\caption{Splitting of the normalizers of maximal tori in $E_8(q)$}\label{t:main:E8} \\
\hline
\No &  Representative $w$ & $|w|$ & Torus $T$ &  \\  
\hline\endfirsthead
\multicolumn{4}{c}{\textit{Table~\ref{t:main:E8} (continued)}} \\
\hline
\No &  Representative $w$ & $|w|$ & Torus $T$ &  \\
\hline\endhead
\hline
\endfoot
\endlastfoot
  1  & $1$ & 1 & $(q-1)^8$ &  -- \\  \hline
  2  & $w_1$ & 2 & $(q-1)^6\times(q^2-1)$ & -- \\  \hline
  3  & $w_1w_2$ & 2 & $(q-1)^4\times(q^2-1)^2$ & -- \\  \hline
  4  & $w_3w_1$ & 3 & $(q-1)^5\times(q^3-1)$ &  -- \\  \hline
  5 & $w_2w_3w_5$ & 2 & $(q-1)^2\times(q^2-1)^3$ &  -- \\  \hline
  6 & $w_1w_3w_5$ & 6 & $(q-1)^3\times(q^2-1)\times(q^3-1)$ &  -- \\  \hline
  7 & $w_1w_3w_4$ & 4 & $(q-1)^4\times(q^4-1)$ &  -- \\ \hline
  8 & $w_1w_4w_6w_{69}$ & 2 & $(q-1)^2\times(q+1)^2\times(q^2-1)^2$ &  -- \\ \hline
  9 & $w_1w_2w_3w_5$ & 6 & $(q-1)^2\times(q^2-1)$&  -- \\ 
    & & & $\times(q+1)(q^3-1)$ & \\ \hline 
  10 & $w_1w_5w_3w_6$ & 3 & $(q-1)^2\times(q^3-1)^2$ & +  \\ \hline
  11 & $w_1w_4w_6w_3$ & 4 & $(q-1)^2\times(q^2-1)\times(q^4-1)$ &  -- \\ \hline
  12 & $w_1w_4w_3w_2$ & 5 & $(q-1)^3\times(q^5-1)$ & +  \\ \hline
  13 & $w_3w_2w_5w_4$ & 6 & $(q-1)^2\times(q^2-1)$ & -- \\ \
     & & & $\times(q-1)(q^3+1)$ & \\ \hline 
  14 & $w_3w_2w_4w_{18}$ & 4 & $(q-1)^2\times(q-1)(q^2+1)^2$&  -- \\ \hline
  15 & $w_1w_5w_3w_6w_2$ & 6 & $(q-1)\times(q^3-1)$ &  + \\ 
     & & & $\times(q+1)(q^3-1)$ & \\ \hline
  16 & $w_1w_4w_6w_3w_{69}$ & 4 & $(q-1)^2\times(q+1)^2\times(q^4-1)$& -- \\ \hline
  17 & $w_1w_4w_5w_3w_{69}$ &  10 & $(q-1)^2\times(q+1)(q^5-1)$ &  + \\ \hline
  18 & $w_1w_4w_6w_3w_5$ & 6  & $(q-1)^2\times(q^6-1)$ & + \\ \hline
  19 & $w_2w_5w_3w_4w_6$ & 8  & $(q-1)^2\times(q^2-1)(q^4+1)$ & -- \\ \hline
  20 & $w_{26}w_5w_4w_3w_2$ & 12 & $(q-1)^2\times(q-1)(q^2+1)(q^3+1)$ & -- \\ \hline
  21 & $w_1w_5w_2w_3w_6w_{69}$ & 3 & $(q-1)\times(q^2+q+1)^2\times(q^3-1)$& + \\ \hline
  22 & $w_1w_4w_6w_3w_5w_{69}$ & 6 & $(q-1)\times(q^3+1)\times(q^3-1)$&  + \\ \
     & & & $\times(q+1)$ & \\ \hline
  23 & $w_1w_4w_6w_3w_2w_5$ & 12 &  $(q-1)\times(q^3-1)(q^4-q^2+1)$ &  +  \\ \hline
  24 & $w_1w_4w_{18}w_3w_2w_6$ & 9 &  $(q-1)\times(q-1)(q^6+q^3+1)$ & + \\ \hline
  25 & $w_1w_4w_{18}w_3w_2w_{48}$ & 6 &  $(q-1)\times(q^2-q+1)$ & +  \\ 
     &  &  &  $\times(q-1)(q^4+q^2+1)$ & +  \\ \hline
  26 & $w_1w_4w_{6}w_3w_7$ & 12 &  $(q-1)\times(q^3-1)\times(q^4-1)$ & -- \\ \hline
  27 & $w_1w_4w_{6}w_2w_3w_7$  & 15 &  $(q-1)\times(q^5-1)(q^2+q+1)$  & + \\ \hline
  28 & $w_3w_2w_4w_{18}w_7$  & 4 &  $(q^2-1)\times((q^2+1)(q-1))^2$ & -- \\ \hline
  29 & $w_1w_4w_{6}w_3w_5w_7$  & 7 &  $(q-1)\times(q^7-1)$ & +  \\ \hline
  30 & $w_{46}w_3w_5w_1w_4w_6$  & 8 &  $(q-1)(q^4+1)\times(q-1)(q^2+1)$ & --  \\ \hline
  31 &  $w_2w_3w_5w_7$ & 2 &  $(q^2-1)^4$ & -- \\ \hline
  32 & $w_{74}w_3w_2w_5w_4$ & 6 &  $(q^2-1)^2\times(q+1)(q^3-1)$ & -- \\ \hline
  33 & $w_8w_1w_4w_6w_3$  & 4 &  $(q^2-1)^2\times(q^4-1)$ & -- \\ \hline
  34 & $w_1w_5w_3w_6w_2w_8$ & 6 &  $(q+1)(q^3-1)\times(q+1)(q^3-1)$ & + \\ \hline
  35 & $w_1w_2w_3w_6w_8w_7$ & 12 & $(q+1)(q^3-1)\times(q^4-1)$ & -- \\ \hline
  36 & $w_1w_4w_3w_7w_6w_8$ & 4 & $(q^4-1)\times(q^4-1)$ & -- \\ \hline
  37 & $w_4w_8w_2w_5w_7w_{120}$ & 4 &  $(q^2-1)^2\times(q^2+1)^2$ & -- \\ \hline
  38 & $w_1w_8w_2w_4w_5w_6$ & 10 &  $(q^2-1)\times(q+1)(q^5-1)$ & + \\   \hline
  39 & $w_1w_2w_4w_6w_5w_7$ & 6 &  $(q^2-1)\times(q^6-1)$ & + \\ \hline
  40 & $w_2w_3w_5w_7w_4w_8$ & 6 &  $(q^2-1)\times(q^6-1)$ & + \\ \hline
  41 & $w_2w_3w_4w_8w_7w_{18}$ & 12 &  $(q-1)(q^2+1)\times(q^2+1)(q^3-1)$ & -- \\ \hline
  42 & $w_2w_3w_4w_5w_6w_8$ & 8 &  $(q^2-1)\times(q^2-1)(q^4+1)$ & -- \\ \hline
  43 & $w_1w_5w_8w_2w_3w_6w_{69}$ & 6 &  $(q^2+q+1)^2\times(q+1)(q^3-1)$ & + \\ \hline
  44 & $w_1w_5w_7w_2w_3w_6w_{8}$  & 30 &  $(q+1)(q^2+q+1)(q^5-1)$ & + \\ \hline
  45 & $w_1w_4w_2w_3w_6w_8w_7$ & 20 &  $(q+1)(q^2+1)(q^5-1)$ & + \\ \hline
  46 & $w_1w_4w_6w_3w_5w_7w_{120}$ & 14 &  $(q+1)(q^7-1)$ & + \\ \hline
  47 & $w_1w_3w_4w_5w_6w_7w_8$ & 8 &  $(q^8-1)$ & + \\ \hline
  48 & $w_2w_4w_5w_6w_7w_8w_{120}$ & 8 &  $(q^2-1)\times(q^2+1)\times(q^4+1)$ & -- \\ \hline
  49 & $w_2w_3w_4w_7w_{120}w_8w_{18}$ & 4 & $(q^2+1)^2\times(q^4-1)$ & -- \\ \hline
  50 & $w_2w_3w_5w_4w_8w_6w_{120}$ & 24 &  $(q+1)(q^3-1)(q^4+1)$ & + \\ \hline
  51 & $w_{26}w_5w_4w_3w_2w_{120}w_8$ & 12 &  $(q^2+1)(q^6-1)$ & + \\ \hline
  52 & $w_1w_4w_6w_3w_2w_5w_8$ & 12 &  $(q^2-1)(q^2+q+1)(q^4-q^2+1)$ & + \\ \hline
  53 & $w_1w_4w_{18}w_3w_2w_6w_8$ & 18 &  $(q^2-1)(q^6+q^3+1)$ & + \\ \hline
  54 & $w_1w_4w_{18}w_3w_2w_{48}w_8$ & 6 &  $(q^2-q+1)^2\times(q+1)(q^3-1)$ & + \\ \hline
  55 & $w_2w_3w_4w_5w_6w_7w_8$ & 12 &  $(q^2-1)(q^6+1)$ & + \\ \hline
  56 & $w_1w_2w_3w_5w_6w_8w_{120}w_{69}$ & 3 &  $(q^2+q+1)^4$ & + \\ \hline
  57 & $w_1w_4w_2w_3w_6w_8w_7w_{120}$ & 5 &  $(q^4+q^3+q^2+q+1)^2$ & +  \\ \hline
  58 & $w_1w_3w_4w_5w_6w_7w_8w_{120}$ & 9 &  $(q^2+q+1)\times(q^6+q^3+1)$ & +  \\ \hline
  59 & $w_2w_3w_4w_7w_{120}w_{18}w_8w_{74}$ & 4 &  $(q^2+1)^4$ & -- \\ \hline
  60 & $w_2w_3w_5w_7w_4w_6w_8w_{114}$ & 12&  $(q^2+1)\times(q^6+1)$ & + \\ \hline
  61 & $w_4w_6w_8w_{113}w_3w_5w_{32}w_7$ & 8 &  $(q^4+1)^2$ & +  \\ \hline
  62 & $w_1w_2w_3w_4w_5w_6w_8w_{120}$& 12 &  $(q^4-q^2+1)(q^2+q+1)$ & +  \\ 
     & & & $\times(q^2+q+1)$  &  \\ \hline
  63 &  $w_1w_4w_6w_8w_3w_{32}w_5w_{120}$ & 6 & $(q^4+q^2+1)\times(q^2+q+1)$ & +  \\ 
     & & & $\times(q^2-q+1)$ & \\ \hline
  64 & $w_1w_2w_3w_4w_5w_6w_7w_8$ & 30 &  $q^8+q^7-q^5-q^4-q^3+q+1$ & + \\ \hline
  65 & $w_1w_4w_6w_8w_{32}w_5w_7w_{120}$ & 24 &  $q^8-q^4+1$ & + \\ \hline
  66 & $w_1w_4w_6w_8w_{32}w_2w_5w_7$ & 20 &  $q^8-q^6+q^4-q^2+1$ & + \\ \hline
  67 & $w_2w_{32}w_5w_7w_1w_4w_6w_{65}$ & 12  &  $(q^4-q^2+1)^2$ & + \\ 
    \hline
\end{longtable}

\section{Results for all types}
In this section we collect the results of Theorems~\ref{th:E6}-\ref{th:E8} in one table. Information for types $E_7$ and $E_8$ is taken from Tables~\ref{t:main:E7} and \ref{t:main:E8}, respectively. 
We use the symbol '+' if the corresponding torus has a complement in its algebraic normalizer, otherwise we put the symbol '--' into the cell. 
To obtain results for $^2E_6(q)$ one can start with the maximal torus $\overline{T}_{\sigma}\simeq\mathbb{Z}_{q+1}^6$ and then use the proofs from~\cite{GS} replacing $q$ by $-q$ everywhere. So we take information for type $E_6^\varepsilon$
from \cite[Table~1]{GS}. The sign '$\pm$' for the case 14 means that
the algebraic normalizers splits over the torus if and only if $q\equiv\varepsilon1\pmod4$.

\begin{longtable}{|c|l|c|c|c||c|l|c|}
\caption{Splitting of normalizers in $E_l(q)$, where $l\in\{6,7,8\}$}\label{t:main} \\
\hline 
Torus & Rep. & $E_6^{\varepsilon}$ & $E_7$  & $E_8$  & Torus & Rep. & $E_8$ \\ \hline
\hline\endfirsthead
\multicolumn{8}{c}{\textit{Table~\ref{t:main} (continued)}} \\
\hline 
Torus & Rep. & $E_6^{\varepsilon}$ & $E_7$  & $E_8$  & Torus & Rep. & $E_8$ \\ \hline\endhead
\endfoot
\endlastfoot
  1 & $1$  & -- & -- & -- & 35 & $\alpha\beta\gamma\zeta\vartheta\eta$ & -- \\ \hline
  2 & $\alpha$ &  -- & -- & -- & 36 & $\alpha\delta\gamma\eta\zeta\vartheta$ & -- \\ \hline
  3 & $\alpha\beta$ &  -- & -- & -- & 37 & $\delta\vartheta\beta\varepsilon\eta\lambda$ &  -- \\ \hline
  4 & $\gamma\alpha$ & + & + & -- & 38 & $\alpha\vartheta\beta\delta\varepsilon\zeta$ &  + \\  \hline
  5 & $\beta\gamma\varepsilon$ & -- & -- & -- & 39 & $\alpha\beta\delta\zeta\varepsilon\eta$ & + \\ \hline
  6 & $\alpha\gamma\varepsilon$ & + & +  & -- & 40 & $\beta\gamma\varepsilon\eta\delta\vartheta$ & + \\ \hline
  7 & $\alpha\gamma\delta$ & -- & -- & -- & 41 & $\beta\gamma\delta\vartheta\eta\rho$ & --  \\ \hline
  8 & $\alpha\delta\zeta\mu$ & -- & -- & -- & 42 & $\beta\gamma\delta\varepsilon\zeta\vartheta$ & -- \\ \hline
  9 & $\alpha\beta\gamma\varepsilon$ & + & + & -- & 43 & $\alpha\varepsilon\vartheta\beta\gamma\zeta\mu$ & + \\ \hline
  10 & $\alpha\varepsilon\gamma\zeta$ & + & + & + & 44 & $\alpha\varepsilon\eta\beta\gamma\zeta\vartheta$ & + \\ \hline
  11 & $\alpha\delta\zeta\gamma$ & -- & -- & -- & 45 & $\alpha\delta\beta\gamma\zeta\vartheta\eta$ & + \\ \hline
  12 & $\alpha\delta\gamma\beta$ & + & + & + & 46 & $\alpha\delta\zeta\gamma\varepsilon\eta\lambda$ & + \\ \hline
  13 & $\gamma\beta\varepsilon\delta$ & + & + & -- & 47 & $\alpha\gamma\delta\varepsilon\zeta\eta\vartheta$ & + \\ \hline
  14 & $\gamma\beta\delta\rho$ & $\pm$ & -- & -- & 48 & $\beta\delta\varepsilon\zeta\eta\vartheta\lambda$ &  -- \\ \hline
  15 & $\alpha\varepsilon\gamma\zeta\beta$ & + & + & + & 49 & $\beta\gamma\delta\eta\lambda\vartheta\rho$ &  -- \\ \hline
  16 & $\alpha\delta\zeta\gamma\mu$ & -- & -- & -- & 50 & $\beta\gamma\varepsilon\delta\vartheta\zeta\lambda$ &  + \\ \hline
  17 & $\alpha\delta\varepsilon\gamma\mu$ & + & + & + & 51 & $\pi\varepsilon\delta\gamma\beta\lambda\vartheta$ &  +  \\ \hline
  18 & $\alpha\delta\zeta\gamma\varepsilon$ & + & + & + & 52 & $\alpha\delta\zeta\gamma\beta\varepsilon\vartheta$ & + \\ \hline
  19 & $\beta\varepsilon\gamma\delta\zeta$ & + & + & -- & 53 & $\alpha\delta\rho\gamma\beta\zeta\vartheta$ & +  \\ \hline
  20 & $\pi\varepsilon\delta\gamma\beta$ & + & + & -- & 54 & $\alpha\delta\rho\gamma\beta\nu\vartheta$ & + \\ \hline
  21 & $\alpha\varepsilon\beta\gamma\zeta\mu$ & + & + & + & 55 & $\beta\gamma\delta\varepsilon\zeta\eta\vartheta$  & + \\ \hline
  22 & $\alpha\delta\zeta\gamma\varepsilon\mu$ & + & + & + & 56 & $\alpha\beta\gamma\varepsilon\zeta\vartheta\lambda\mu$  & + \\ \hline
  23 & $\alpha\delta\zeta\gamma\beta\varepsilon$ & + & + & + & 57 & $\alpha\delta\beta\gamma\zeta\vartheta\eta\lambda$ & + \\ \hline
  24 & $\alpha\delta\rho\gamma\beta\zeta$ & + & + & + & 58 & $\alpha\gamma\delta\varepsilon\zeta\eta\vartheta\lambda$  & + \\ \hline
  25 & $\alpha\delta\rho\gamma\beta\nu$ &  + & + & + & 59 & $\beta\gamma\delta\eta\lambda\rho\vartheta\kappa$ & -- \\ \hline
  26 & $\alpha\delta\zeta\gamma\eta$ &  & + & -- & 60 & $\beta\gamma\varepsilon\eta\delta\zeta\vartheta\psi$ & + \\ \hline
  27 & $\alpha\delta\zeta\beta\gamma\eta$ &  & + & + & 61 & $\delta\zeta\vartheta\omega\gamma\varepsilon\xi\eta$ & + \\ \hline
  28 &  $\gamma\beta\delta\rho\eta$ & & -- & -- & 62 & $\alpha\beta\gamma\delta\varepsilon\zeta\vartheta\lambda$ & + \\ \hline
  29 &  $\alpha\delta\zeta\gamma\varepsilon\eta$ &  & + & + & 63 & $\alpha\delta\zeta\vartheta\gamma\xi\varepsilon\lambda$ & + \\ \hline
  30 & $\tau\gamma\varepsilon\alpha\delta\zeta$ & & + & -- & 64 & $\alpha\beta\gamma\delta\varepsilon\zeta\eta\vartheta$ & + \\ \hline
  31 & $\beta\gamma\varepsilon\eta$ & & & -- & 65 & $\alpha\delta\zeta\vartheta\xi\varepsilon\eta\lambda$ &  + \\ \hline
  32 & $\kappa\gamma\beta\varepsilon\delta$ & & & -- & 66 & $\alpha\delta\zeta\vartheta\xi\beta\varepsilon\eta$ & +  \\ \hline
  33 & $\vartheta\alpha\delta\zeta\gamma$ & & & -- & 67 & $\beta\xi\varepsilon\eta\alpha\delta\zeta\upsilon$ & + \\  \hline
  34 & $\alpha\varepsilon\gamma\zeta\beta\vartheta$ & & & + & & & \\  \hline
\end{longtable}

\section*{Acknowledgement}
This research was supported by the Russian Science Foundation (project no. 14-21-00065).

\Addresses
\end{document}